\documentclass[11pt,reqno]{amsart}
\usepackage[utf8]{inputenc}
\usepackage{amscd}
\usepackage{amssymb}
\usepackage{amsmath}
\usepackage{amsthm}
\usepackage{color}
\usepackage{graphicx}
\usepackage{wrapfig}
\usepackage{mathtools}
\usepackage[colorlinks]{hyperref}
\usepackage{theoremref}
\usepackage{caption}
\usepackage{comment}
\usepackage{thmtools}

\newtheorem{theorem}{Theorem}[section]
\newtheorem{proposition}[theorem]{Proposition}
\newtheorem{lemma}[theorem]{Lemma}
\newtheorem{corollary}[theorem]{Corollary}

\theoremstyle{definition}
\newtheorem{definition}[theorem]{Definition}
\newtheorem{example}[theorem]{Example}
\newtheorem{remark}[theorem]{Remark}

\renewcommand{\phi}{\varphi}
\renewcommand{\i}{\mspace{1.5mu}\mathbf i}
\renewcommand{\j}{\mspace{1.5mu}\mathbf j}
\renewcommand{\k}{\mspace{1.5mu}\mathbf k}
\newcommand{\n}{\mathbf n}
\renewcommand{\v}{\mathbf v}
\newcommand{\lk}{\operatorname{lk}}
\newcommand{\tr}{\operatorname{tr}}

\newcommand{\sign}{\operatorname{sign}}

\renewcommand{\Re}{\operatorname{Re}}
\renewcommand{\Im}{\operatorname{Im}}
\newcommand{\Hom}{\operatorname{Hom}}

\newcommand{\SU}{\operatorname{SU}}
\newcommand{\SO}{\operatorname{SO}}
\newcommand{\Sp}{\operatorname{Sp}}

\thanks{Both authors were partially supported by the NSF Grant DMS-1952762}
\author[Zedan Liu]{Zedan Liu}
\author[Nikolai Saveliev]{Nikolai Saveliev}
\address{Department of Mathematics \newline\indent University of Miami, PO Box 249085 \newline\indent Coral Gables, FL 33124}
\email{\rm{zedan.liu@miami.edu}}
\email{\rm{saveliev@math.miami.edu}}

\begin{document}
\begin{abstract}
The B{\'e}nard--Conway invariant of links in the 3-sphere is a Casson--Lin type invariant defined by counting irreducible SU(2) representations of the link group with fixed meridional traces. For two-component links with linking number one, the invariant has been shown to equal a symmetrized multivariable link signature. We extend this result to all two-component links with non-zero linking number. A key ingredient in the proof is an explicit calculation of the B{\'e}nard--Conway invariant for $(2,2\ell)$--torus links with the help of the Chebyshev polynomials.
\end{abstract}
\title{The Bénard--Conway invariant of two-component links}

\maketitle

\begin{sloppypar}

\section{Introduction}
	
The practice of defining invariants of links in $3$-manifolds using unitary representations of the link group has a long history. Xiao-Song Lin \cite{Lin} defined an invariant of knots in $S^3$ by counting irreducible $\SU(2)$ representations of the knot group sending the meridians to zero-trace matrices. Herald \cite{Herald} and Heusener and Kroll \cite{Heusener-Kroll} extended this construction by allowing matrices of a fixed trace which is not necessarily zero. The construction was further extended to links of more than one component by Harper and Saveliev \cite{Harper-Saveliev} and Boden and Harper \cite{Boden-Harper} by counting projective unitary representations. These invariants are closely related to gauge theory: for example, Floer homology theories of Daemi--Scaduto \cite{Daemi-Scaduto} and Kronheimer--Mrowka \cite{Kronheimer-Mrowka} can be viewed as categorifying the invariants of Lin and Harper--Saveliev, respectively. {\par}

The latest in this line of link invariants is the multivariable Casson--Lin type invariant $h_L$ of Bénard and Conway \cite{Benard-Conway}. It is defined for colored links $L \subset S^3$ by counting irreducible $\SU(2)$ representations of the link group sending the meridians to matrices of a fixed trace away from the roots of the multivariable Alexander polynomial. While this invariant is defined for links $L$ of any number of components, it is best studied for two-component links. In particular, for an oriented ordered link $L = L_1 \cup L_2 \subset S^3$ with linking number $\lk(L_1, L_2) = 1$, Bénard and Conway \cite[Theorem 1.1]{Benard-Conway} identify $h_L$ with a symmetrized multi\-variable link signature $\sigma_L$ of Cimasoni and Florens \cite{Cimasoni-Florens}. 
	
The purpose of this paper is to extend the result of Bénard and Conway to arbitrary oriented ordered links $L_1 \cup L_2 \subset S^3$ of two components with linking number $\lk(L_1, L_2) \neq 0$. To be precise, we prove the following theorem. 
	
\begin{theorem}\thlabel{thm:main-theorem-introduction}
Let $L = L_1 \cup L_2 \subset S^3$ be a two-component oriented ordered link with $\lk(L_1, L_2) \neq 0$ and, for any choice of $(\alpha_1, \alpha_2) \in (0, \pi)^2$, denote $(\omega_1, \omega_2) = (e^{2 \i \alpha_1}, e^{2 \i \alpha_2})$. If the multivariable Alexander polynomial of $L$ satisfies $\Delta_L (\omega_1^{\epsilon_1}, \omega_2^{\epsilon_2}) \neq 0$ for all possible $\epsilon_1,\epsilon_2 = \pm 1$ then
\begin{equation}\label{E:main}
h_L(\alpha_1, \alpha_2) \; = \; - \frac{1}{2} \left(\sigma_L \left( \omega_1, \omega_2 \right) + \sigma_L \left( \omega_1, \omega^{- 1}_2 \right) \right).
\end{equation}
\end{theorem}

The proof of this theorem consists of two parts, just like the proof of \cite[Theorem 1.1]{Benard-Conway}. The first part shows that a crossing change within an individual component of $L$ changes both sides of the formula of \thref{thm:main-theorem-introduction} by the same amount. This fact was only proved in \cite[Theorem 1.1]{Benard-Conway} for links with $\lk(L_1, L_2) = 1$, but that proof easily extends to links with $\lk(L_1, L_2) \neq 0$, as we explain in Section \ref{sec:induction}. After changing enough crossings within individual components of $L$, we only need to check that equation \eqref{E:main} holds for just one representative in each link homotopy class of $L$. According to Milnor \cite{Milnor1}, link homotopy classes of two component links are completely characterized by $\lk(L_1, L_2)$, therefore, it is sufficient to prove \thref{thm:main-theorem-introduction} for $(2, 2 \ell)$--torus links $L_{\ell}$ with $\lk(L_1, L_2) = \ell \neq 0$. 

This second part of the proof occupies Sections \ref{sec:count-rep}, \ref{sec:five} and \ref{sec:six} of the paper. We compute the Bénard--Conway invariant $h_{L_{\ell}}$ directly from its definition and compare the answer with the symmetrized Cimasoni--Florens signature. The invariant $h_{L_{\ell}}$, whose definition we recall in Section \ref{sec:two}, is in an intersection number of two oriented curves in a $2$-dimensional orbifold, traditionally referred to as a pillowcase. We come up with a parameterization of the pillowcase, in which the intersecting curves are given by explicit equations in terms of the Chebyshev polynomials; see Theorem \ref{T:chebyshev}. Checking the transversality and computing the intersection signs is then accomplished by a straightforward calculation. 

Note that the above argument fails for links $L = L_1 \cup L_2$ with $\lk (L_1,L_2) = 0$ because the base case, which is the link $L_{\ell}$ with $\ell = 0$, has a vanishing Alexander polynomial and hence its B{\'e}nard--Conway invariant is not defined. 

The following result follows from Theorem \ref{thm:main-theorem-introduction} and the properties of the Cimasoni--Florens signature \cite{Cimasoni-Florens}. It is proved in Section \ref{sec:seven} together with \thref{thm:main-theorem-introduction}.

\begin{theorem}\thlabel{thm:second-theorem-introduction}
For any two-component link $L = L_1 \cup L_2$ as in the statement of Theorem \ref{thm:main-theorem-introduction}, the invariant $h_L (\alpha_1, \alpha_2)$ is independent of the orientation of the link $L$. Moreover, $h_L (\pi/2,\pi/2)$ equals minus the Murasugi signature \cite{Murasugi} of the link $L$.  
\end{theorem}

Finally, we wish to mention how our work is related to that of Daemi and Scaduto \cite{Daemi-Scaduto-2}. Let $L$ be a link in $S^3$ of any number of components with non-zero determinant. Daemi and Scaduto  \cite{Daemi-Scaduto-2} define irreducible instanton homology $I(L)$ as a $\mathbb Z/4$ graded abelian group which, in favorable circumstances, is generated at the chain level by the conjugacy classes of irreducible $\SU(2)$ representations of $\pi_1 (S^3 - L)$ sending meridians to zero-trace matrices. In general, a perturbation may be necessary to achieve transversality. One expects that the Euler characteristic of $I(L)$ equals, up to an overall constant, the Benard--Conway invariant $h_L (\pi/2,\ldots,\pi/2)$. Daemi and Scaduto \cite{Daemi-Scaduto-2} further show that the Euler characteristic of $I(L)$ is proportional to the Murasugi signature of the link $L$, which matches our Theorem 1.2 in the special case of two-component links $L$ with $\alpha_1 = \alpha_2 = \pi/2$. 

\medskip\noindent
{\bf Acknowledgments:} We thank Hans Boden, Anthony Conway, Daniel Ruberman, and Chris Scaduto for useful discussions and sharing their expertise. 

	
\section{Preliminaries}\label{sec:two}
In this section, we will recall the definitions of the Cimasoni--Florens signature \cite{Cimasoni-Florens} and the B{\'e}nard--Conway invariant \cite{Benard-Conway} for oriented links in the 3-sphere.
		

\subsection{The Cimasoni--Florens signature}
A \textit{$\mu$--colored link} $L \subset S^3$ is an oriented link whose components are partitioned into sublinks $L = L_1\,\cup \ldots \cup\, L_{\mu}$. A \textit{C-complex} for a $\mu$-colored link $L$ is a union $S = S_1 \cup \dots \cup S_{\mu}$ of oriented surfaces in $S^3$ such that
\begin{enumerate}
\item for all $i$, $S_i$ is a Seifert surface for $L_i$ (possibly disconnected, but with no closed components),
\item for all $i \ne j$, $S_i \cap S_j$ is either empty or a union of clasps (see \autoref{fig:clasp}), and 
\item for all $i, j, k$ which are pairwise distinct, $S_i \cap S_j \cap S_k$ is empty.
\end{enumerate}

\medskip
\begin{figure}[!ht]
\centering
\includegraphics[width=0.98\textwidth]{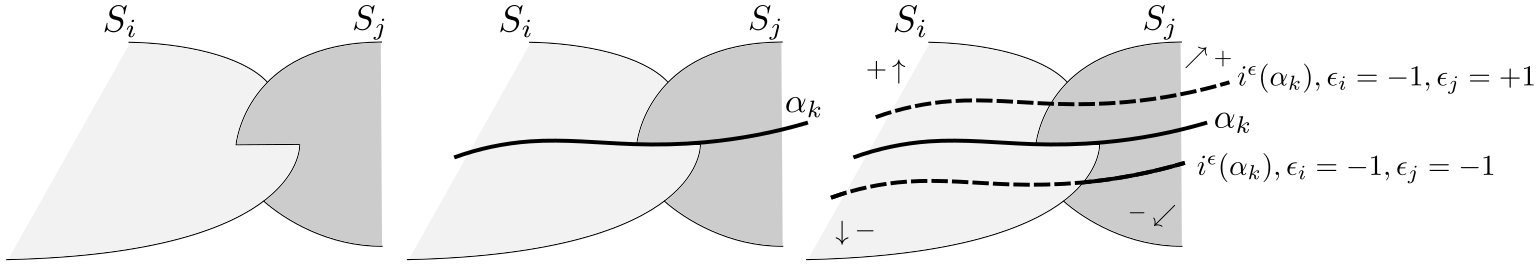}
\caption{A clasp, a loop crossing clasp, and examples of push offs}
\label{fig:clasp}
\end{figure}
		
Such a C-complex exists for any colored link, see \cite[Lemma 1]{Cimasoni} for the proof. A $1$-cycle in a C-complex is called a \textit{loop} if it is an oriented simple closed curve which behaves as illustrated in \autoref{fig:clasp} whenever it crosses a clasp. There exists a collection of loops whose homology classes form a basis of $H_1(S)$. Define $i^{\epsilon}: H_1(S) \to H_1(S^3 \setminus S)$ by setting $i^{\epsilon}([x])$ equal the class of the $1$-cycle obtained by pushing $x$ in the $\epsilon_i$--normal direction off $S_i$ for $i = 1, \dots, \mu$; see \autoref{fig:clasp}. Let
		\[
			\alpha^{\epsilon}: H_1(S) \times H_1(S) \to \mathbb{Z}
		\]
		be the bilinear form given by $\alpha^{\epsilon}(x, y) = \lk(i^{\epsilon}(x), y)$ where $\lk$ denotes the linking number. Fix a basis of $H_1(S)$ and denote by $A^{\epsilon}$ the matrix of $\alpha^{\epsilon}$. If $\mu = 1$, then $\alpha^-$ is the usual Seifert form and $A^-$ is the usual Seifert matrix. Also, note that $A^{- \epsilon} = (A^{\epsilon})^{\top}$ for any $\epsilon$. {\par}
		
Given a $\mu$-colored link $L$ and a C-complex $S$, fix a basis in $H_1 (S)$ and consider the associated Seifert matrices $A^{\epsilon}$. Let $A(t_1, \dots, t_{\mu})$ be the matrix with coefficients in $\mathbb{Z}[t_1, \dots, t_{\mu}]$ defined by 
		\[
			A(t_1, \dots, t_{\mu})\, =\, \sum\limits_{\epsilon}\; \epsilon_1 \cdots \epsilon_{\mu}\; t_1^{\frac{1 - \epsilon_1}{2}} \cdots t_{\mu}^{\frac{1 - \epsilon_{\mu}}{2}} A^{\epsilon},
		\]
where summation runs over the $2^{\mu}$ possible sequences $\epsilon = (\epsilon_1, \dots, \epsilon_{\mu})$ of $(\pm 1)$'s. For $\omega = (\omega_1, \dots, \omega_{\mu}) \in (S^1)^{\mu} \subset \mathbb{C}^{\mu}$, set
\begin{equation}\label{E:matrix H}
H(\omega)\, =\, \prod_{i = 1}^{\mu}\; (1 - \overline{\omega}_i) A(\omega).
\end{equation}
Since $A^{- \epsilon} = (A^{\epsilon})^{\top}$, the matrix $H(\omega)$ is Hermitian, hence its eigenvalues are real. Define the \textit{signature} $\sign H(\omega)$ as the number of positive eigenvalues minus the number of negative eigenvalues. This signature does not depend on the choice of basis of $H_1(S)$, nor the choice of the C-complex of the link, thus it gives a well-defined isotopy invariant of the link $L$; see \cite[Theorem 2.1]{Cimasoni-Florens} for the proof.
		
		\begin{definition}
			Let $L$ be a $\mu$-colored link and $T_*^{\mu} = (S^1 \setminus \lbrace 1 \rbrace)^{\mu}$. The \textit{Cimasoni--Florens signature} of $L$ is the function $\sigma_L: T_*^{\mu} \to \mathbb{Z}$ given by $\sigma_L(\omega) = \sign H(\omega)$.
		\end{definition}


\subsection{Colored links and colored braids}
Before we go on to define the B{\'e}nard--Conway invariant of colored links, we will interpret the latter as the closures of colored braids; see Murakami \cite{Murakami}.

Recall that the (Artin) braid group $B_n$ on $n$ strands is the finitely presented group with $(n-1)$ generators $\sigma_1, \sigma_2, \dots, \sigma_{n - 1}$ subject to the relations $\sigma_i \sigma_{i + 1} \sigma_i = \sigma_{i + 1} \sigma_i \sigma_{i + 1}$ for each $i$, and $\sigma_i \sigma_j = \sigma_j \sigma_i$ for $|i - j| > 1$. Geometrically, a generator $\sigma_i$ can be viewed as the isotopy class of the braid whose $(i+1)$-st strand crosses over the $i$-th strand. The \textit{closure} $\widehat{\beta}$ of a braid $\beta \in B_n$ is the link obtained from $\beta$ by connecting the lower endpoints of the braid and its upper endpoints with parallel strands. The link $\widehat\beta$ is canonically oriented by choosing the downward orientations on the strands of $\beta$. 

A \textit{$\mu$--colored braid} is a braid $\beta \in B_n$ together with an assignment to each of its strands of an integer (called the \textit{color}) in $\lbrace 1, 2, \dots, \mu \rbrace$ via a surjective map. A $\mu$--colored braid induces \textit{$\mu$--colorings} $c$ and $c'$ on the upper and lower endpoints of the braid, which are $n$--tuples of integers in $\lbrace 1,2,\ldots,\mu \rbrace$. For any $\mu$--coloring $c$, the $\mu$--colored braids with $c' = c$ form a \textit{colored braid group} $B_c \subset B_n$. For example, if $\mu = 1$, then $c = (1,\ldots,1)$ and $B_c$ is simply the braid group $B_n$, and if $\mu = n$ and $c = (1, 2, \dots, n )$ then $B_c$ is the pure braid group on $n$ strands. 

The \textit{closure} $\widehat{\beta}$ of a $\mu$--colored braid $\beta\in B_c$,  obtained from $\beta$ by connecting the lower endpoints of the braid with the upper endpoints with colored parallel strands, is a $\mu$--colored link. A colored version of Alexander's theorem states that every $\mu$--colored link is the closure $\widehat \beta$ of a $\mu$--colored braid $\beta \in B_c$, and a colored version of Markov's theorem determines when two $\mu$--colored braids have isotopic closures.


\subsection{The B{\'e}nard--Conway invariant}\label{sec:invariant-def}
Consider the Lie group $\SU(2)$ of two--by--two unitary matrices with determinant one. We will be identifying it with the Lie group $\Sp(1)$ of unit quaternions via
\begin{equation}\label{E:sp1}
\begin{pmatrix}
a & b \\
- \overline{b} & \overline{a}
\end{pmatrix} \; \mapsto \; a + b \j
\end{equation}

\smallskip\noindent
and using the language of matrices and unit quaternions interchangeably. 

Let $F_n$ be a free group on $n$ generators $x_1,\ldots, x_n$. The group $B_n$ acts naturally on $F_n$ via 
\begin{equation}
			\label{eq:group-action}
			x_j \sigma_i = \begin{cases}
				\; x_i x_{i + 1} x_i^{- 1}, & j = i, \\
				\; x_i, & j = i + 1, \\
				\; x_j, & \text{otherwise.}
			\end{cases}
		\end{equation}

\smallskip\noindent
This action induces an action on the representation space $R(F_n) = \Hom(F_n, \SU(2)) = \SU(2)^n$. More concretely, every braid $\beta \in B_n$ induces a homeomorphism $\beta: \SU(2)^n \to \SU(2)^n$ by the rule $(X_1,\ldots, X_n)\beta = (X_1\beta,\ldots, X_n\beta)$. For example, $(X_1, X_2)\,\sigma_1 = (X_1 X_2 X_1^{-1}, X_1)$ for the generator $\sigma_1 \in B_2$.

Let now $L = L_1\,\cup\ldots\cup\,L_{\mu}$ be an oriented $\mu$--colored link. Represent it as the closure of a $\mu$--colored braid $\beta \in B_c$ on $n$ strands. Given a $\mu$--tuple $\alpha = (\alpha_1,\ldots,\alpha_{\mu}) \in (0,\pi)^{\mu}$, consider the representation space
\[
R_n^{\alpha, c} = \big\lbrace \left( X_1, X_2, \dots, X_n \right) \in \SU(2)^n \; \big| \; \tr(X_i) = 2 \cos(\alpha_{c_i})\; \text{ for }\; i = 1, \ldots, n \, \big\rbrace.
\]
Since the trace is preserved by conjugation, the homeomorphism $\beta: \SU(2)^n \to \SU(2)^n$ constructed above restricts to a homeomorphism $\beta: R_n^{\alpha, c} \to R_n^{\alpha, c}$ with the graph 
\[
\Gamma_{\beta}^{\alpha,c} = \big\lbrace \left( X_1, X_2, \dots, X_n, X_1 \beta, X_2 \beta, \dots, X_n \beta \right) \in R_n^{\alpha, c} \times R_n^{\alpha, c} \big\rbrace.
\]
Note that the trivial braid in $B_c$ gives rise to the graph which is just the diagonal 
\[
\Lambda_n^{\alpha, c} = \big\lbrace \left( X_1, X_2, \dots, X_n, X_1, X_2, \dots, X_n \right) \in R_n^{\alpha, c} \times R_n^{\alpha, c} \big\rbrace.
\]
Since the product $X_1\cdots X_n$ is preserved by the action of $\beta$ (see formula \eqref{eq:group-action}), one immediately concludes that both $\Gamma_{\beta}^{\alpha,c}$ and $\Lambda_n^{\alpha, c}$ are subspaces of the ambient space 
\[
H_n^{\alpha, c} = \big\lbrace \left( X_1, X_2, \dots, X_n, Y_1, Y_2, \dots, Y_n \right) \in R_n^{\alpha, c} \times R_n^{\alpha, c} \; \big| \; X_1\cdots X_n\, =\, Y_1\cdots Y_n\, \big\rbrace.
\]
The group $\SO(3)$ acts via conjugation on the spaces $\Gamma_{\beta}^{\alpha,c}$, $\Lambda_n^{\alpha, c}$ and $H_n^{\alpha, c}$. We will restrict this action to irreducible representations, where it is free, and denote the quotient spaces by $\widehat{\Gamma}_{\beta}^{\alpha,c}$, $\widehat{\Lambda}_n^{\alpha, c}$ and $\widehat{H}_n^{\alpha, c}$. These are smooth open manifolds of dimensions 
\[
\dim \widehat{\Gamma}_{\beta}^{\alpha,c}\, =\, 2n - 3,\quad \dim \widehat{\Lambda}_n^{\alpha, c}\, =\, 2n - 3,\quad \dim \widehat{H}_n^{\alpha, c}\, =\, 4n - 6;
\]
see \cite[Lemma 3.4]{Benard-Conway}. 

The key observation now is that the points of the intersection $\widehat{\Lambda}_n^{\alpha, c}\,\cap\,\widehat{\Gamma}_{\beta}^{\alpha,c} \subset \widehat{H}_n^{\alpha, c}$ are precisely the conjugacy classes of irreducible $\SU(2)$ representations of the link group 
\begin{equation}\label{E:presentation}
\pi_1 \big(S^3 - L\,\big)\, =\, \big\langle\; x_1,\ldots,x_n\; \big|\; x_i = x_i\beta\; \text{ for }\; i = 1,\ldots, n\;\big\rangle
\end{equation}
sending the meridians $x_i$ to matrices $X_i \in \SU(2)$ of trace $\tr (X_i) = 2\cos(\alpha_{c_i})$. From this point on, the definition of the B{\'e}nard--Conway invariant proceeds by making sense of the intersection number of $\widehat{\Lambda}_n^{\alpha, c}$ and $\widehat{\Gamma}_{\beta}^{\alpha,c}$ in $\widehat{H}_n^{\alpha, c}$. We will briefly outline the procedure and refer to \cite{Benard-Conway} for detailed proofs. 

Let $\Delta_L(t_1,\ldots, t_{\mu})$ be the multivariable Alexander polynomial of $L$, and, given a $\mu$--tuple $\alpha = (\alpha_1,\ldots, \alpha_{\mu})\, \in\, (0, \pi)^{\mu}$, consider the finite set 
\begin{equation}\label{E:S}
S(\alpha)\, =\, \big\lbrace  \big(e^{\epsilon_1 2 \i \alpha_1}, \dots, e^{\epsilon_{\mu} 2 \i \alpha_{\mu}}\big)
\; \big| \; \epsilon_i = \pm 1\;\text{ for }\; i = 1, \dots, \mu \big\rbrace.
\end{equation}

\begin{proposition}\thlabel{prop:compactness}
If $\Delta_{\widehat{\beta}}\,(\omega) \neq 0$ for all $\mu$--tuples $\omega \in S(\alpha)$ then the intersection $\widehat{\Lambda}_n^{\alpha, c}\,\cap\,\widehat{\Gamma}_{\beta}^{\alpha,c} \subset \widehat{H}_n^{\alpha, c}$ is compact.
\end{proposition}
		
Let $\alpha = (\alpha_1,\ldots, \alpha_{\mu})$ satisfy the condition of \thref{prop:compactness}. Since $\widehat{\Lambda}_n^{\alpha,c}\,\cap\,\widehat{\Gamma}_{\beta}^{\alpha, c} \subset \widehat{H}_n^{\alpha, c}$ is compact, the graph $\widehat{\Gamma}_{\beta}^{\alpha,c}$ can be perturbed if necessary using a perturbation with compact support to make the intersection $\widehat{\Lambda}_n^{\alpha,c}\,\cap\,\widehat{\Gamma}_{\beta}^{\alpha, c} \subset \widehat{H}_n^{\alpha, c}$ transversal and hence a compact $0$-dimensional manifold.

Next, we will orient all the manifolds in question. Denote by $\mathbb{S}(\theta)$ the conjugacy class of matrices in $\SU(2)$ with the trace $2\cos \theta$. Assuming that $0 < \theta < \pi$, the conjugacy classes $\mathbb S(\theta)$ are naturally homeomorphic to each other and to the standard 2-sphere $S^2$. Choose an (arbitrary) orientation on $S^2$. The space $R_n^{\alpha, c}$ is a product of the spheres $\mathbb{S}(\alpha_{c_i})$, $i = 1,\ldots, n$, and we will endow it with the product orientation. The spaces $\Lambda_{n}^{\alpha, c}$ and $\Gamma_{\beta}^{\alpha,c}$, which are diffeomorphic to $R_n^{\alpha, c}$ via projection onto the first $n$ factors, will be endowed with the induced orientations. To orient $H_n^{\alpha, c}$, consider the map $f_n: R_n^{\alpha, c} \times R_n^{\alpha, c} \to \SU(2)$ given by $f_n(X_1, \dots, X_n, Y_1, \dots, Y_n) = X_1 \cdots X_n\, (Y_1 \cdots Y_n)^{- 1}$. Observe that $H_n^{\alpha, c} = f_n^{- 1}(1)$, so that we can pull back the canonical orientation of $\SU(2)$ to obtain an orientation on $H_n^{\alpha, c}$. Since the adjoint action of $\SO(3)$ on each $\mathbb{S}(\theta)$ is orientation preserving, we can endow $\widehat{\Lambda}_n^{\alpha, c}$, $\widehat{\Gamma}_{\beta}^{\alpha,c}$, and $\widehat{H}_n^{\alpha, c}$ with the induced quotient orientation. 
		
The intersection number of $\widehat{\Lambda}_n^{\alpha, c}$ and $\widehat{\Gamma}_{\beta}^{\alpha,c}$ will be denoted by $\big\langle\widehat{\Lambda}_n^{\alpha, c}, \widehat{\Gamma}_{\beta}^{\alpha,c}\big\rangle_{\widehat{H}_n^{\alpha, c}}$. The following proposition ensures that it only depends on the isotopy class of the closure of $\beta$.
		
\begin{proposition}\thlabel{prop:markov-moves}
Under the assumptions of \thref{prop:compactness}, the intersection number $\big\langle \widehat{\Lambda}_n^{\alpha, c}, \widehat{\Gamma}_{\beta}^{\alpha,c}\big\rangle_{\widehat{H}_n^{\alpha, c}}$ is preserved by the Markov moves.
\end{proposition}
		
We will summarize the above construction in the following definition. 	
	
\begin{definition}\thlabel{def:invariant-def}
Let $L$ be a $\mu$--colored link in $S^3$ and $\alpha \in (0, \pi)^{\mu}$ a $\mu$--tuple. Let $\beta \in B_c$ be a colored braid of $n$ strands whose closure is $L$. Suppose that $\Delta_L(\omega_{\epsilon}) \ne 0$ for all $\omega_{\epsilon} \in S(\alpha)$. Then the \textit{Bénard--Conway invariant} of $L$ is well-defined by the formula
\begin{equation}\label{E:def}
h_L(\alpha)\, =\, \big\langle \widehat{\Lambda}_n^{\alpha, c}, \widehat{\Gamma}_{\beta}^{\alpha,c}\big\rangle_{\widehat{H}_n^{\alpha, c}}.
\end{equation}
\end{definition}
		
		
\section{Inductive step for two-component links}\label{sec:induction}
From now on, we will restrict ourselves to 2-colored links $L$ of two components, which are just ordered two-component links. In this section, we will investigate what happens to the two sides of the formula \eqref{E:main} under a crossing change within a single component of $L$. 

\begin{theorem}\thlabel{thm:inductive-step}
Under the assumptions of \thref{thm:main-theorem-introduction}, the two sides of formula \eqref{E:main} stay well-defined and change by the same amount when a crossing change occurs within one of the components of the link.
\end{theorem}

The rest of this section will be dedicated to the proof of Theorem \ref{thm:inductive-step}. Given a pair $\alpha = (\alpha_1,\alpha_2) \in (0,\pi)^2$, the set defined in \eqref{E:S} becomes 
\[
S(\alpha)\, =\, \big\lbrace \big(e^{\epsilon_1 2\i\alpha_1}, e^{\epsilon_2 2\i\alpha_2}\big) \; \big| \; \epsilon_1 = \pm 1,\; \epsilon_2 = \pm 1 \big\rbrace.
\]
In addition, define the sets 
\[
S_j(\alpha) = \big\lbrace \big(e^{\epsilon_1 2\i\alpha_1}, e^{\epsilon_2 2\i\alpha_2}\big) \in S(\alpha) \; \big| \; \epsilon_j = 1 \big\rbrace,\;\; j = 1, 2,
\]
and denote by $\nabla_L (t_1,t_2)$ the multivariable Conway potential function. Recall that $\nabla_L (t_1, t_2)$ equals $\Delta_L \big(t_1^2, t_2^2\big)$ up to multiplication by the powers of $\pm t_1$ and $\pm t_2$.

\begin{lemma}\thlabel{lmm:inductive-lemma}
Let $L = L_1 \cup L_2$ be a two-component oriented ordered link and denote by $L_+$ the link obtained from $L$ by a negative crossing change within a component $L_j$ of the link $L$. Suppose that $\alpha = (\alpha_1, \alpha_2) \in (0, \pi)^2$ is such that, for all $\omega = (\omega_1, \omega_2) \in S_j(\alpha)$, one has $\omega_1^2 \ne 1$, $\omega_1^2 \ne 1$ and $\omega_1\omega_2 \ne 1$. If $\nabla_L(\omega^{1/2}) \ne 0$ and $\nabla_{L_+}(\omega^{1/2}) \ne 0$ then
\[
h_{L_+}(\alpha) - h_L(\alpha) = \#\, \big\lbrace \omega \in S_j(\alpha) \; \big| \; \nabla_{L_+}(\omega^{1/2})\, \nabla_L(\omega^{1/2}) < 0\big\rbrace.
\]
\end{lemma}
	
A proof of this lemma can be found in \cite[Proposition 5.10 and Remark 5.11]{Benard-Conway}.
	
\begin{lemma}
Let $L = L_1 \cup L_2$ be an ordered two-component link with $\lk(L_1\cup L_2) \neq 0$ and suppose that $\omega \in (S^1 \setminus \lbrace 1 \rbrace)^2$ is not a root of $\Delta_L(t_1,t_2)$. Then
\[
\sigma_L(\omega)\; \equiv\; 2 + \lk(L_1, L_2) + \sign(\nabla_L(\omega^{1/2})) \; \bmod{4}.
\]
In addition, suppose that $L_+$ is obtained from $L$ by a negative crossing change within one of the components of $L$, and that $\omega$ is not a root of either $\Delta_{L_+}(t_1,t_2)$ or $\Delta_L(t_1,t_2)$. Then $\sigma_{L_+}(\omega) - \sigma_L(\omega)$ is either $0$ or $-2$.
\end{lemma}

A proof of this lemma can be found in \cite[Lemma 6.2]{Benard-Conway}. The following corollary is an easy consequence of the above lemma.
	
\begin{corollary}\thlabel{cor:inductive-cor}
Let $L$ be an ordered two-component link with $\lk(L_1\cup L_2) \neq 0$ and $L_+$ a link obtained from $L$ by a negative crossing change within one of its components. Suppose that $\omega \in (S^1 \setminus \lbrace 1 \rbrace)^2$ is such that $\nabla_L(\omega^{1/2}) \ne 0$ and $\nabla_{L_+}(\omega^{1/2}) \ne 0$. Then
\[
\sigma_{L_+}(\omega) - \sigma_L(\omega)\; =\; \begin{cases} \quad 0, & \text{if }\; \nabla_{L_+}(\omega^{1 / 2}) \nabla_L(\omega^{1 / 2}) > 0 , \\ \; - 2, & \text{if }\; \nabla_{L_+}(\omega^{1 / 2}) \nabla_L(\omega^{1 / 2}) < 0 . \end{cases}
\]
\end{corollary}

We are now ready to sketch a proof of \thref{thm:inductive-step}, which will be a straightforward modification of the proof of Theorem 6.4 in \cite{Benard-Conway}. That theorem is first proved under the additional assumption that $\alpha_1$ and $\alpha_2$ are transcendental. This assumption is then removed by showing that the invariant $h_L (\alpha)$ is locally constant in $\alpha$. Only the first part of the proof needs to be modified.

Observe first that a crossing change within one component of $L$ does not make the multivariable Alexander polynomial vanish, which ensures that the B{\'e}nard--Conway invariant of $L_+$ is well-defined. This follows from the Torres formula \cite{Torres}
\[
\Delta_{L_1\cup L_2} (t_1,1)\;\; \dot{=}\;\; \frac{t_1^{\,\lk(L_1,L_2)} - 1}{t_1 - 1} \cdot \Delta_{L_1} (t_1)
\]
and our assumption that $\lk (L_1,L_2) \neq 0$. Next, we will show that the two sides of the equation \eqref{E:main} change by the same amount under the crossing change. Assume without loss of generality that $L_+$ is obtained from $L = L_1 \cup L_2$ by a negative crossing change in $L_1$. According to \thref{lmm:inductive-lemma}, we have
\[
h_{L_+}(\alpha) - h_L(\alpha) = \# \big\lbrace \omega \in S_1(\alpha) \; \big| \; \nabla_{L_+}(\omega^{1 / 2}) \nabla_L(\omega^{1 / 2}) < 0\big\rbrace.
\]
Note that $S_1(\alpha)$ has exactly two elements, $(\omega_1, \omega_2)$ and $(\omega_1, \omega_2^{-1})$, where $\omega_1 = e^{2\i\alpha_1}$ and $\omega_2 = e^{2\i\alpha_2}$. It now follows from \thref{cor:inductive-cor} that
\[
\begin{aligned}
h_{L_+}(\alpha) - h_L(\alpha) &= -\, \frac{1}{2}\,(\sigma_{L_+}(\omega_1, \omega_2) - \sigma_L(\omega_1, \omega_2))\, -\, \frac{1}{2}\,(\sigma_{L_+}(\omega_1, \omega_2^{-1}) - \sigma_L(\omega_1, \omega_2^{-1})) \\ &= -\, \frac{1}{2}\,(\sigma_{L_+}(\omega_1, \omega_2)+\sigma_{L_+}(\omega_1, \omega_2^{-1}))\, +\, \frac{1}{2}\,(\sigma_L(\omega_1, \omega_2)+\sigma_L(\omega_1, \omega_2^{-1})),
\end{aligned}
\]
which completes the proof. 


\section{$\SU(2)$ representations of torus links}\label{sec:count-rep}
To complete the proof of Theorem \thref{thm:main-theorem-introduction}, it is sufficient to verify the formula \eqref{E:main} for $(2,2\ell)$--torus links $L_{\ell}$ with $\ell \neq 0$. Our convention here is that $\ell > 0$ gives the right--handed torus link and $\ell < 0$ the left--handed one. The verification will take up the rest of the paper. We begin in this section by describing the irreducible $\SU(2)$ representations of the link group of $L_{\ell}$ with fixed meridional traces.


\subsection{Geometry of $\SU(2)$}
We will continue to identify $\SU(2)$ matrices with unit quaternions as in \eqref{E:sp1}. Any $q \in \SU(2)$ can then be written in the form $q = \cos(\alpha) + \sin(\alpha) Q = e^{\alpha Q}$, where $\alpha \in [0, \pi]$ and $Q$ is a purely-imaginary unit quaternion. This expression is unique except when $q = \pm 1$. We will refer to $2 \cos(\alpha)$ as the trace of $q$ and to $\cos(\alpha)$ as the real part of $q$. Using arbitrary unit quaternions, $q = \cos(\alpha) + \sin(\alpha) Q$ can be conjugated to $\cos(\alpha) + \sin(\alpha) \i = e^{\alpha \i}$. Using only unit complex numbers, $q = \cos(\alpha) + \sin(\alpha) Q$ can be conjugated to $\cos(\alpha) + \sin(\alpha) \big(\cos(\beta) \i + \sin(\beta) \j \,\big)$ for some $\beta \in [0, \pi]$. Alternatively, it can be expressed as $\cos(\gamma) (\cos(\beta) + \sin(\beta) \i) + \sin(\gamma) \j$ for $\gamma \in [0, \pi]$ and $\beta \in [0,2 \pi]$ with $\cos(\gamma) \cos(\beta) = \cos(\alpha)$ since the trace is conjugation invariant.


\subsection{Counting the representations}\label{S:reps}
Let us first assume that $\ell > 0$ and consider the following presentation of the link group
\[
\pi_1(S^3 \setminus L_{\ell}) = \big\langle\, x_1, x_2\; \big|\; (x_1 x_2)^{\ell} = (x_2x_1)^{\ell} \,\big\rangle,
\]
where $x_1$ and $x_2$ are the meridians of the two components of $L_{\ell}$. For a fixed choice of $(\alpha_1, \alpha_2) \in (0, \pi)^2$, we wish to describe the conjugacy classes of irreducible representations $\rho: \pi_1(S^3 \setminus L_{\ell}) \to \SU(2)$ sending the meridians of the components of $L_{\ell}$ to matrices with the respective traces $2\cos(\alpha_j)$, $j = 1, 2$.

Since $(x_2 x_1)^{\ell} = (x_2 x_1)^{\ell} x_2 x_2^{-1} = x_2 (x_1 x_2)^{\ell} x_2^{-1}$, the relation $(x_1 x_2)^{\ell} = (x_2x_1)^{\ell}$ is equivalent to $(x_1 x_2)^{\ell}$ commuting with $x_2$. Therefore, we are looking for non-commuting unit quaternions $\rho(x_1)$ and $\rho(x_2)$ with prescribed traces such that $(\rho(x_1) \rho(x_2))^{\ell}$ commutes with $\rho(x_2)$. 

Conjugate $\rho$ so that $\rho(x_2) = \cos(\alpha_2) + \sin(\alpha_2) \i$ and $\rho(x_1) = \cos(\alpha_1) + \sin(\alpha_1) (\cos(\phi) \i + \sin(\phi) \j)$ for some $\phi \in (0, \pi)$. Since $\rho(x_2) \neq \pm 1$, the quaternions $(\rho(x_1) \rho(x_2))^{\ell}$ and $\rho(x_2)$ commute only if the former is a complex number. In turn, this implies that $(\rho(x_1) \rho(x_2))^{\ell} = \pm 1$ because otherwise $\rho(x_1)\rho(x_2)$ is a complex number and $\rho$ is reducible. 

We end up looking for non-commuting unit quaternions $\rho(x_1)$ and $\rho(x_2)$ with prescribed traces such that $\rho(x_1) \rho(x_2)$ is an $\ell$--th root of $\pm 1$ (different from $\pm 1$ because otherwise $\rho$ is again reducible). The latter condition means that the real part of $\rho(x_1) \rho(x_2)$ equals $\cos(\pi m/\ell)$ or, equivalently, 
\begin{equation}\label{eq:solution-phi}
 \cos(\alpha_1) \cos(\alpha_2) - \sin(\alpha_1) \sin(\alpha_2) \cos(\phi)\; = \; \cos(\pi m/\ell), \quad m = 1, 2, \dots, \ell - 1.
 \end{equation}

\begin{proposition}\label{P:solution-phi}
Given a $(2, 2 \ell)$--torus link $L_{\ell}$ with $\ell \neq 0$ and a choice of $(\alpha_1, \alpha_2) \in (0, \pi)^2$, the number of the conjugacy classes of irreducible representations $\pi_1 (S^3 \setminus L_{\ell}) \to \SU(2)$ equals the number of integers $m \in \{1,\ldots, |\ell| -1\}$ such that 
\[
\cos(\pi m/\ell) \in [\,\cos(\alpha_1 + \alpha_2), \cos(\alpha_1 - \alpha_2)\,].
\]
\end{proposition}

\begin{proof}
Let us assume that $\ell > 0$ since the argument for $\ell < 0$ is similar. Use trigonometric identities to express \eqref{eq:solution-phi} in terms of $\alpha_1 - \alpha_2$ and $\alpha_1 + \alpha_2$ as $\cos (\alpha_1 - \alpha_2) \sin^2 (\phi/2) + \cos (\alpha_1 + \alpha_2) \cos^2 (\phi/2)\; = \; \cos(\pi m/\ell)$ or, equivalently,
\[
\cos (\alpha_1 - \alpha_2) (1-t) + \cos (\alpha_1 + \alpha_2)\, t\; = \; \cos(\pi m/\ell),
\] 
where $t = \cos^2 (\phi/2) \in [0,\pi]$. The latter equation has a unique solution $t \in [0,1]$, and therefore the equation \eqref{eq:solution-phi} has a unique solution $\phi \in (0,\pi)$, if and only if $\cos(\pi m/\ell) \in [\,\cos(\alpha_1-\alpha_2), \cos(\alpha_1 + \alpha_2)\,]$. 
\end{proof} 

\begin{example}
The cases of $\ell = 1$ and $\ell = -1$ correspond to the two oriented Hopf links. In both cases, the link group has no irreducible $\SU(2)$ representations. The former case served as the base of induction in \cite{Benard-Conway}. The count of conjugacy classes for $\ell = 3$ is depicted in \autoref{fig:count-rep}. The diamond shapes in the figure have to do with the Alexander polynomial of the link as discussed in detail in the following subsection.
\end{example}
	
\begin{figure}[!ht]
\centering
\includegraphics[width=0.6\textwidth]{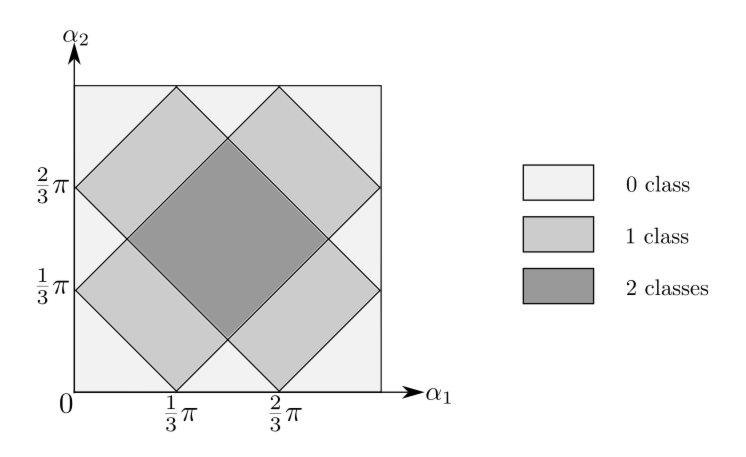}
\caption{Counting conjugacy classes, $\ell = 3$}
\label{fig:count-rep}
\end{figure}


\subsection{The Alexander polynomial}
The number of the conjugacy classes of representations in Proposition \ref{P:solution-phi} jumps with the  change in $\alpha_1$ and $\alpha_2$. The following proposition shows that these jumps occur at the roots of the multivariable Alexander polynomial, where the B{\'e}nard--Conway invariant is not defined.

\begin{proposition}\thlabel{prop:non-vanishing-condition}
Let $L_{\ell}$ be a $(2,2\ell)$--torus link with $\ell \neq 0$ and $\Delta_{L_{\ell}}(t_1,t_2)$ its multivariable Alexander polynomial. Given $(\alpha_1,\alpha_2) \in (0,\pi)^2$, let $(\omega_1, \omega_2) = (e^{2 \i \alpha_1}, e^{2 \i \alpha_2})$. Then $\Delta_{L_{\ell}}(\omega_1^{\pm 1}, \omega_2^{\pm 1}) = 0$ exactly when 
\begin{equation}\label{E:roots}
\alpha_1 + \alpha_2 = \frac {\pi m}{| \ell |}\; \text{ and }\; \alpha_1 - \alpha_2 + \pi = \frac{\pi m}{| \ell |}\;\text{ with }\; 0 < m < |\ell| \;\text{ and }\; |\ell| < m < 2 | \ell |.
\end{equation}
\end{proposition}

\begin{proof}
According to Milnor \cite{Milnor2}, the Alexander polynomial of $L_{\ell} $ is given by the formula 
\medskip
\[
\Delta_{L_{\ell}}(t_1,t_2)\;\; \dot{=}\;\; \frac{(t_1 t_2)^{|\ell|} - 1}{t_1 t_2 - 1}.
\]

\medskip\noindent
Therefore, $\Delta_{L_{\ell}}(\omega_1^{\epsilon_1}, \omega_2^{\epsilon_2}) = 0$ exactly when $\omega_1^{\epsilon_1} \omega_2^{\epsilon_2}$ is an $|\ell|$--th root of unity that is not $1$, which is equivalent to $\epsilon_1 \alpha_1 + \epsilon_2 \alpha_2\, =\, \pi m/|\ell|$ for some $0 < m < |\ell|$, up to addition or subtraction of $\pi$. The result now follows. 
\end{proof}
	
\begin{remark}
When $\ell = 0$, the torus link $L_{\ell}$ is just the unlink of two components. Its multivariable Alexander polynomial is identically zero hence its B{\'e}nard--Conway invariant is not defined. This is the reason behind our assumption that $\ell \neq 0$.
\end{remark}
	
	
\section{The pillowcase and intersection theory}\label{sec:five}
In this section, we will compute the B{\'e}nard--Conway invariants of $(2,2\ell)$--torus links as the intersection number of the manifolds $\widehat{\Lambda}_n^{\alpha, c}$ and $\widehat{\Gamma}_{\beta}^{\alpha,c}$ inside of $\widehat{H}_n^{\alpha, c}$. This will involve counting the representations described in Proposition \ref{P:solution-phi} with plus/minus signs, after making sure that the intersection in question is transversal. 


\subsection{The setup}
Let $L_{\ell}$ be a $(2,2\ell)$--torus link and assume that $\ell$ is positive. The case of negative $\ell$ can be treated in a similar manner, and both cases will be discussed in detail when we perform explicit calculations in Section \ref{sec:sign}. Let $\alpha = (\alpha_1, \alpha_2) \in (0, \pi)^2$ be such that $\alpha_1 + \alpha_2 \ne {\pi m}/{\ell}$ and $\alpha_1 - \alpha_2 + \pi \ne {\pi m}/{\ell}$ for any $0 < m < \ell$ and $\ell < m < 2\ell$. This condition guarantees that the B{\'e}nard--Conway invariant of $L_{\ell}$ is well-defined; see Proposition \ref{prop:non-vanishing-condition}. We will view the link $L_{\ell}$ as a $2$-colored link which is the closure of the $2$--colored braid $\beta = \sigma_1^{2 \ell} \in B_c$ with the $2$--coloring $c = (1,2)$. The B{\'e}nard--Conway invariant of $L_{\ell}$ is then the intersection number of $\widehat{\Lambda}_2^{\alpha, c}$ and $\widehat{\Gamma}_{\beta}^{\alpha,c}$ inside of $\widehat{H}_2^{\alpha, c}$. Our first goal will be to parameterize these manifolds.


\subsection{The pillowcase}
Recall that $\widehat{H}_2^{\alpha, c}$ is an open 2-manifold obtained by removing the conjugacy classes of reducible representations from the orbifold $H_2^{\alpha, c}/\SO(3)$, where
\[
H_2^{\alpha, c}\, =\, \big\lbrace(X_1, X_2, Y_1, Y_2)\in R_2^{\alpha, c} \times R_2^{\alpha, c} \; \big| \; X_1 X_2 = Y_1 Y_2 \big\rbrace.
\]
In the special case of $\alpha_1 = \alpha_2 = \pi/2$, the orbifold $H_2^{\alpha, c}/\SO(3)$ is referred to as a pillowcase; see for instance Lin \cite{Lin}. We will extend this terminology to the general case. 

After conjugation, we may assume that $X_2 = e^{\alpha_2 \i}$ and $X_1 = e^{\alpha_1 P_1}$, where $P_1 = \i e^{- \k \phi}$ with $\phi \in [0, \pi]$; see Section \ref{S:reps}. Write $Y_1 = e^{\alpha_1 Q_1}$ and $Y_2 = e^{\alpha_2 Q_2}$, where $Q_1$ and $Q_2$ are purely imaginary unit quaternions. For any given $Y_1$, the equation $X_1 X_2 = Y_1 Y_2$ can be uniquely solved for $Y_2$ if and only if the trace of $Y_1^{- 1} X_1 X_2$ matches that of $Y_2$, that is,
\[
\Re \big(e^{- \alpha_1 Q_1}\cdot e^{\alpha_1 \i e^{- \k \phi}} \cdot e^{\alpha_2 \i}\big)\; =\; \cos (\alpha_2).
\]
Write $Q_1 = u \i + v \j + w \k$, where $u^2 + v^2 + w^2 = 1$, then the above equation is equivalent to
\begin{align}
\big( \sin(\alpha_1) \cos(\alpha_2) \cos(\phi) + \cos(\alpha_1) \sin(\alpha_2) \big) & u \label{eq:plane-Q1}  \\
+ \sin(\alpha_1) \cos(\alpha_2) \sin(\phi) & v \notag \\
-  \sin(\alpha_1) \sin(\alpha_2) \sin(\phi) & w = \sin(\alpha_1) \cos(\alpha_2) + \cos(\alpha_1) \sin(\alpha_2) \cos(\phi).\notag
\end{align}

We will handle the case of $\phi \in (0, \pi)$ first. In this case, equation \eqref{eq:plane-Q1} has the form $\n \cdot (u, v, w) = d$, where $d\, =\, \sin(\alpha_1) \cos(\alpha_2) + \cos(\alpha_1) \sin(\alpha_2) \cos(\phi)$ and
\medskip
\[
\n\, =\, \begin{pmatrix}
\sin(\alpha_1) \cos(\alpha_2) \cos(\phi) + \cos(\alpha_1) \sin(\alpha_2) \\
\sin(\alpha_1) \cos(\alpha_2) \sin(\phi) \\
- \sin(\alpha_1) \sin(\alpha_2) \sin(\phi)\quad
\end{pmatrix}
\] 

\medskip\noindent
is a non-zero vector (because its third coordinate is never zero). Therefore, equation \eqref{eq:plane-Q1} describes a plane in the space $\mathbb{R}^3$ of purely imaginary quaternions, and our task becomes describing the intersection of this plane with the unit sphere $S^2$ of purely imaginary unit quaternions given by the equation $u^2 + v^2 + w^2 = 1$. Note that the point on the plane closest to the origin is given by $d \n / | \n |^2$. One can easily see that the distance from this point to the origin equals 
\begin{equation}\label{E:dist}
\frac d {|\n|}\; = \; \frac{d}{\sqrt{\,d^2 + \sin^2 (\alpha_2) \sin^2 (\phi)}}
\end{equation}

\smallskip\noindent
and that this distance is strictly less than one, making the intersection of the plane with the unit sphere $S^2$ a circle $S^1_{\phi}$ for any choice of $\phi \in (0,\pi)$. Therefore, away from the points with $\phi = 0$ and $\phi = \pi$, the orbifold $H_2^{\alpha, c}/\SO(3)$ is homeomorphic to a cylinder $S^1 \times (0,\pi)$. It can be parameterized as follows. 

\begin{figure}[!ht]
		\centering
		\includegraphics[width=0.9\textwidth]{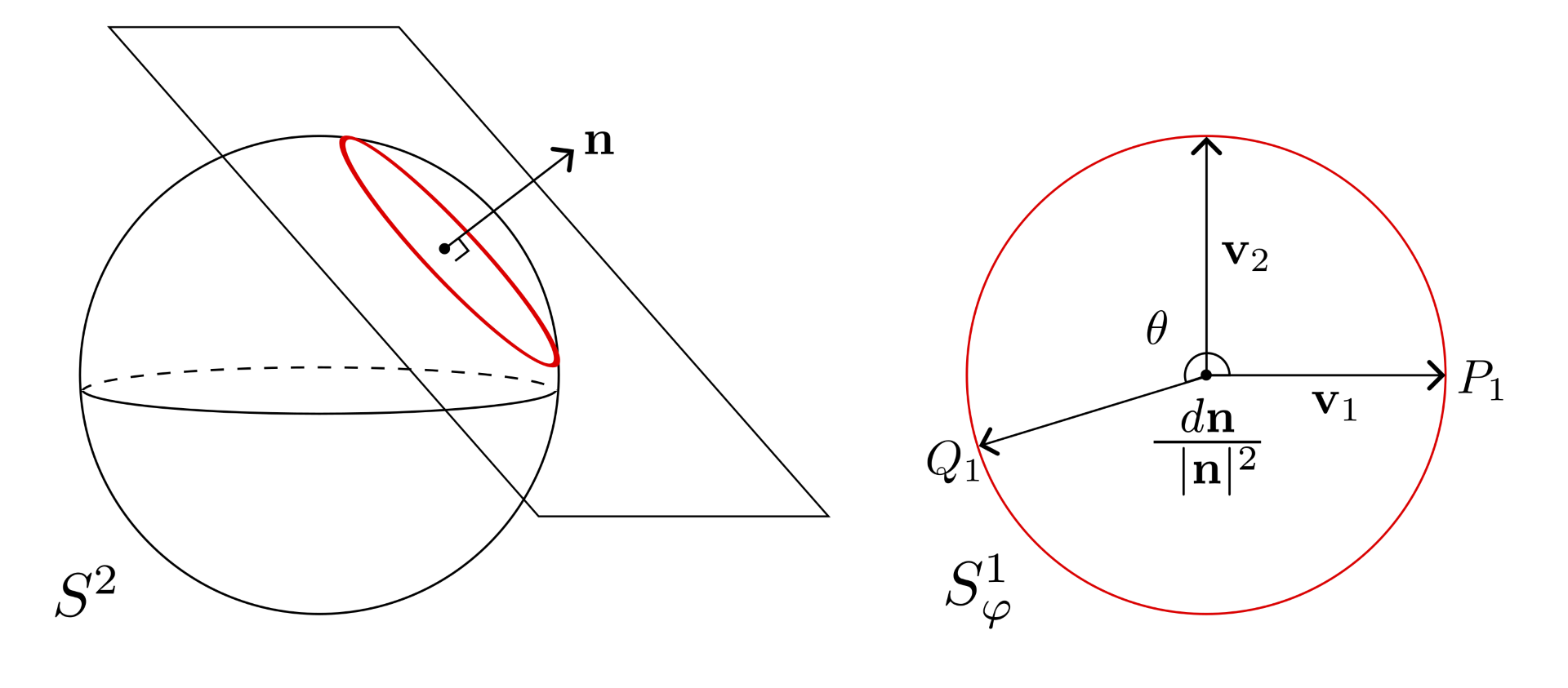}
		\caption{Vectors $\v_1$, $\v_2$ and points $P_1$, $Q_1$ on the circle $S^1_{\phi}$}
		\label{fig:P1-Q1}
	\end{figure}

Note that the circle $S^1_{\phi}$ contains the vector $P_1 = \i e^{- \k \phi} =\cos(\phi) \i + \sin(\phi) \j$, which corresponds to the solution with $(X_1,X_2,X_1,X_2)$ of the equation \eqref{eq:plane-Q1} . Let 
\[
\v_1\; =\; P_1 - \frac{d \n}{\;\, |\n|^2}\quad\text{and}\quad \v_2\; =\; \frac{\n}{|\n|}\; \times\; \v_1.
\]
 Since the vectors $\v_1$ and $\v_2$ are orthogonal to each other, we can write 
 \begin{equation}\label{E:param}
 Q_1\; =\; \frac{d \n}{\;\,|\n|^2}\, +\, \cos \theta\cdot \v_1\, +\, \sin \theta \cdot \v_2
 \end{equation}
for some $\theta \in [0, 2 \pi]$; see \autoref{fig:P1-Q1}. Therefore, away from the points with $\phi = 0$ and $\phi = \pi$, the cylinder $H_2^{\alpha, c}/\SO(3)$ is parameterized by $(\phi,\theta) \in (0, \pi)\, \times\, [0, 2 \pi]$ with the intervals $\theta = 0$ and $\theta = 2 \pi$ identified; see \autoref{fig:hat-H2}. 
\bigskip

\begin{figure}[!ht]
\centering
\includegraphics[width=0.6\textwidth]{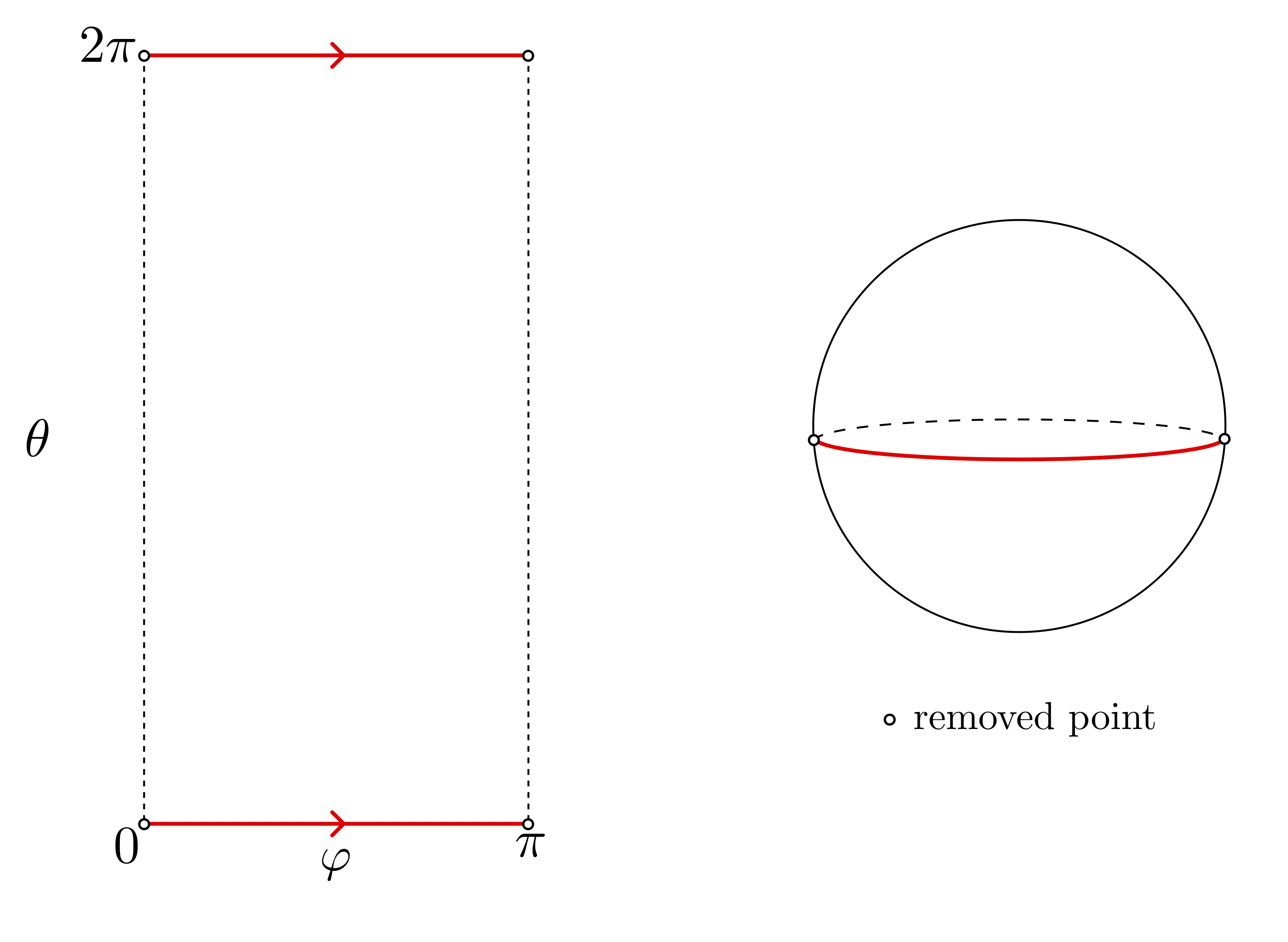}
\caption{Parametrization by $(\phi,\theta)$}
\label{fig:hat-H2}
\end{figure}

We now turn to the full space $H_2^{\alpha, c}/\SO(3)$. It is a compactification of the cylinder described above by the points with $\phi = 0$ and $\phi = \pi$. The type of compactification one obtains depends on $\alpha_1$ and $\alpha_2\colon$

\begin{figure}[!ht]
\includegraphics[width=0.6\textwidth]{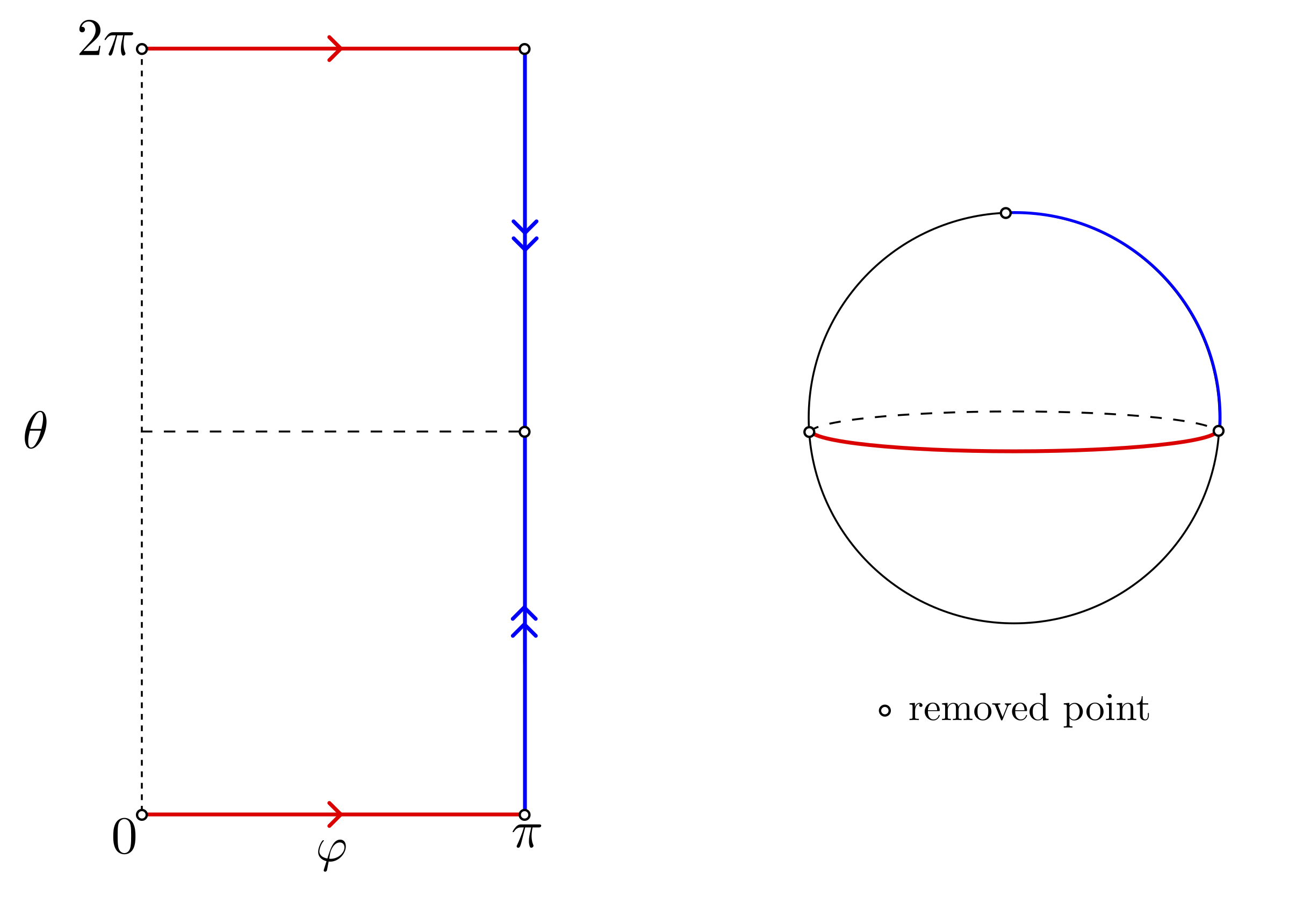}
\caption{Special case when $\alpha_1 = \alpha_2 \ne \pi / 2$}
\label{fig:hat-H2-special}
\end{figure}

\smallskip \textbf{(1)}\; For a generic choice of $\alpha_1$ and $\alpha_2$ with $\sin(\alpha_1 \pm \alpha_2) \neq 0$,  the above calculation extends to the points with $\phi = 0$ and $\phi = \pi$ giving rise in each case to a single reducible representation. In this case, the cylinder compactifies to a 2-sphere, and $\widehat{H}_2^{\alpha, c}$ is a 2-sphere with two points removed; see \autoref{fig:hat-H2}.

\smallskip \textbf{(2)}\; The case of $\alpha_1 = \alpha_2 \ne \pi / 2$, in which $\sin(\alpha_1 - \alpha_2) = 0$ but $\sin(\alpha_1 + \alpha_2) \ne 0$, was investigated in \cite{Heusener-Kroll}. In this case, all representations with $\phi = \pi$ and $\theta \in (0, 2 \pi)$ are irreducible and the points $(\pi,\theta)$ and $(\pi,2\pi - \theta)$ are identified; see \autoref{fig:hat-H2-special}. Topologically, $\widehat{H}_2^{\alpha, c}$ is a 2-sphere with three points removed. 

\smallskip \textbf{(3)}\; In the special case of $\alpha_1 \ne \alpha_2$ and $\alpha_1 + \alpha_2 = \pi$ we have $\sin(\alpha_1 + \alpha_2) = 0$ but $\sin(\alpha_1 - \alpha_2) \ne 0$. This time around, the representations with $\phi = 0$ and $\theta \in (0, 2 \pi)$ are irreducible and $\widehat{H}_2^{\alpha, c}$ is a mirror image of \autoref{fig:hat-H2-special}. 

\smallskip \textbf{(4)}\; The remaining case of $\alpha_1 = \alpha_2 = \pi / 2$ is the original case investigated by Lin \cite{Lin}. The representation with $\phi = 0$ and $\phi = \pi$ are now irreducible for all $\theta \in (0, 2 \pi)$. The resulting $\widehat{H}^{\alpha, c}_2$ is a sphere with four points removed.

\smallskip
All of the resulting orbifolds $H_2^{\alpha, c}/\SO(3)$ will be referred to as pillowcases. The calculation that follows will be the same for all of them because it happens away from the points with $\phi = 0$ and $\phi = \pi$.


\subsection{Intersection theory in the pillowcase}
The diagonal $\widehat{\Lambda}_2^{\alpha, c}$ is a subspace of $\widehat{H}_2^{\alpha, c}$ given by the equation $X_1 = Y_1$. In our parameterization of the pillowcase, $\widehat{\Lambda}_2^{\alpha, c}$ is then exactly the subspace where $P_1 = Q_1$ or $\theta = 0$. It is shown in red in \autoref{fig:hat-H2}. 

Next, we will parametrize $\widehat{\Gamma}_{\beta}^{\alpha,c}$. Since $\theta$ is the angle between the vectors $P_1 - d\n/|\n|^2$ and $Q_1 - d\n/|\n|^2$, we have 
\[
\cos(\theta) = \frac{(P_1 - d\n/|\n|^2) \cdot (Q_1 - d\n/|\n|^2)}{1 - d^2/|\n |^2}.
\]	
Using formula \eqref{E:dist}, this simplifies to 
\begin{equation}\label{eq:parametrize-Gamma-1}
\cos(\theta)\; =\; \frac{(|\n|^2 P_1 - d\n) \cdot Q_1}{\sin^2(\alpha_2) \sin^2(\phi)}.
\end{equation}

\smallskip\noindent
The lemmas that follow simplify this formula further and eventually lead to Theorem \ref{T:chebyshev}, which identifies the right hand side of \eqref{eq:parametrize-Gamma-1} as a Chebyshev polynomial. 

\begin{lemma}\label{L:one}
The right hand side of formula \eqref{eq:parametrize-Gamma-1} is a polynomial in $\cos (\phi)$, which will be denoted by $P(\cos(\phi))$.
\end{lemma}

\begin{proof}
We will proceed by simplifying \eqref{eq:parametrize-Gamma-1} while keeping track of its dependence on $\phi$. By a direct calculation,
\begin{equation}\label{E:P}
|\n|^2 P_1 - d \n\; =\; \begin{pmatrix}
			\sin^2(\phi) \cdot A_1\; \\
			\sin(\phi) \cdot A_2 \\
			\sin(\phi) \cdot A_3
\end{pmatrix},
\end{equation}

\medskip\noindent
where $A_1$, $A_2$ and $A_3$ are real valued polynomials in $\cos(\phi)$ of degrees $\deg A_1 = 1$, $\deg A_2 = 2$ and $\deg A_3 = 1$. To compute $Q_1$, recall that $Y_1$ is the image of $X_1$ under the action of the braid $\beta = \sigma_1^{2 \ell}$, where $(X_1, X_2)\,\sigma_1 = (X_1 X_2 X_1^{- 1}, X_1)$; see Section \ref{sec:invariant-def}. Therefore, 
\[
(X_1, X_2)\,\sigma_1^{2 \ell} \, = \, \big( (X_1 X_2)^{\ell} \cdot X_1 \cdot (X_1 X_2)^{- \ell}, (X_1 X_2)^{\ell} \cdot X_2 \cdot (X_1 X_2)^{- \ell} \big)
\]
and 
\[
Q_1 = (X_1 X_2)^{\ell} \cdot P_1 \cdot (X_1 X_2)^{- \ell}.
\]
Write
\begin{multline*}
X_1 X_2 \, = \, (\cos(\alpha_1) + \sin (\alpha_1)\cos(\phi)\i + \sin(\alpha_1)\sin(\phi)\j)\mspace{1mu}(\cos(\alpha_2) +\sin(\alpha_2) \i)\, = \\ 
(\cos(\alpha_1) \cos(\alpha_2) - \sin(\alpha_1) \sin(\alpha_2) \cos (\phi)) +
(\cos(\alpha_1) \sin(\alpha_2) + \sin(\alpha_1) \cos(\alpha_2) \cos (\phi))\i \\
+\, \sin(\alpha_1)\cos (\alpha_2) \sin (\phi)\j - \sin(\alpha_1)\sin(\alpha_2)\sin(\phi)\k 
\, =\,  U_1 + \sin(\phi)\cdot V_1 \j,
\end{multline*}
where $U_1$ and $V_1$ are complex valued polynomials in $\cos(\phi)$ of degrees $\deg U_1 = 1$ and $\deg V_1 = 0$. Using an induction on $\ell$ one can easily see that 
\[
(X_1 X_2)^{\ell} \, = \, U_{\ell}\, +\, \sin(\phi)\cdot V_{\ell}\, \j,
\]
where $U_{\ell}$ and $V_{\ell}$ are complex valued polynomials in $\cos(\phi)$ of degrees $\deg U_{\ell} = \ell$ and $\deg V_{\ell} = \ell - 1$. A straightforward calculation then shows that
\begin{equation}\label{E:Q}
Q_1 = \begin{pmatrix}
\;\qquad\quad\; B_1 \\
\;\sin(\phi) \cdot B_2 \\
\;\sin(\phi) \cdot B_3
\end{pmatrix},
\end{equation}
where $B_1$, $B_2$ and $B_3$ are real valued polynomials in $\cos(\phi)$ of degrees $\deg B_1 = 2\ell + 1$, $\deg B_2 = 2\ell$ and $\deg B_3 = 2\ell$. By taking the dot product of \eqref{E:P} and \eqref{E:Q} we conclude that the numerator of \eqref{eq:parametrize-Gamma-1} has the form $\sin^2(\phi)$ times a polynomial in $\cos(\phi)$ of degree at most $2\ell + 2$. The factors of $\sin^2(\phi)$ in the numerator and the denominator of \eqref{eq:parametrize-Gamma-1} cancel, thereby finishing the proof.
\end{proof}

\begin{lemma}\label{L:two}
The polynomial $P(\cos(\phi))$ has degree $2\ell$ and its leading coefficient equals $2^{2\ell -1} \sin^{2\ell} (\alpha_1) \sin^{2\ell}(\alpha_2)$.
\end{lemma}

\begin{proof}
We will follow the proof of the previous lemma but make our calculations more precise. One can check using induction on $\ell$ that
\begin{equation}\label{E:UV2}
\i\mspace{1mu} \overline U_{\ell}\, = \, \cos(\phi) \cdot V_{\ell}\, +\, \i \cot(\alpha_1)\cdot V_{\ell} + \ldots
\end{equation}
where the dots stand for a polynomial in $\cos(\phi)$ of degree at most $\ell - 2$. A tedious but direct calculation of the polynomials in formulas \eqref{E:P} and \eqref{E:Q} then yields the following formulas\,:
\begin{align*}
A_1\; & = \quad \sin^2(\alpha_1) \sin^2(\alpha_2) \cos\,(\phi) - \sin(\alpha_1)\sin(\alpha_2)\cos(\alpha_1)\cos(\alpha_2)\\
A_2\; & = - \sin^2(\alpha_1) \sin^2(\alpha_2) \cos^2(\phi) + \sin(\alpha_1)\sin(\alpha_2)\cos(\alpha_1)\cos(\alpha_2)\cos(\phi) + \sin^2(\alpha_2) \\
A_3\; & = \quad \sin(\alpha_1) \cos(\alpha_1) \sin^2(\alpha_2)\cos(\phi) + \sin^2(\alpha_1)\sin(\alpha_2)\cos(\alpha_2) 
\end{align*}
and 
\begin{align*}
B_1\, & =\, \big[\,\big(|U_{\ell}|^2 + |V_{\ell}|^2 \cos^2(\phi)\big)\cos(\phi) - 2\Im \big(U_{\ell} \overline V_{\ell}\big) \cos^2(\phi)\,\big]_{2 \ell + 1} \\ & \tag*{$-\;\big[\, |V_{\ell}|^2 \cos(\phi) - 2\Im \big(U_{\ell} \overline V_{\ell}\big)\, \big]_{2\ell - 1}$\qquad} \\
B_2\, & =\, \big[- \Re \big(V_{\ell}^2\big)\cos^2(\phi) + 2\Im \big(U_{\ell} V_{\ell}\big) \cos(\phi) +  \Re \big(U_{\ell}^2\big)\,\big]_{2 \ell}\quad + \big[\,\Re \big(V_{\ell}^2\big)\,\big]_{2\ell - 2}\\
B_3\, & =\, \big[- \Im \big(V_{\ell}^2\big) \cos^2 (\phi) - 2 \Re \big (U_{\ell}V_{\ell}\big)\cos(\phi) +  \Im \big(U_{\ell}^2\big)\,\big]_{2\ell - 1} + \big[\,\Im \big(V_{\ell}^2\big)\,\big]_{2\ell - 2}
\end{align*}
The brackets in the formulas for $B_1$, $B_2$ and $B_3$ contain polynomials whose degrees are indicated by the subscripts. Note that the first bracket in the formula for $B_3$ has degree $2\ell - 1$ rather than $2\ell$, which follows from \eqref{E:UV2}. Another lengthy calculation shows that
\begin{align*}
A_1 B_1 + A_2 B_2 + A_3 B_3\, =\,  2\sin^2(\alpha_1) \sin^2(\alpha_2) \cos^2 (\phi) \big(\Im \big(U_{\ell}\big) - \cos(\phi) \Re \big(V_{\ell}\big) & \big)^2 \\ 
\qquad  - \, 2 \sin(\alpha_1) \sin(\alpha_2) \cos(\alpha_1) \cos(\alpha_2) \cos (\phi) \big(\Im \big(U_{\ell}\big) - \cos(\phi) \Re \big(V_{\ell}\big)&  \big)^2  \\
 \qquad +\, 2 \sin^2 (\alpha_1) \cos^2(\phi) \big( \Im(V_{\ell}) & \big)^2 +\, \ldots
\end{align*}
where the dots stand for a polynomial in $\cos(\phi)$ of degree at most $2\ell - 1$. Using formula \eqref{E:UV2}, this further simplifies to 
\[
A_1 B_1 + A_2 B_2 + A_3 B_3\, =\,2 \cos^2(\phi) \sin^2(\alpha_2) \big(\Im(V_{\ell})\big)^2 +\, \ldots
\]
which is a polynomial of degree $2\ell$. It follows that $P(\cos(\phi))$ is a polynomial of degree $2\ell$ whose  leading term matches that of the polynomial $2 \cos^2(\phi) \big(\Im \big(V_{\ell}\big)\big)^2 $. Another induction shows that the leading term of  $\Im\big( V_{\ell} \big)$ equals 
\[
(-1)^{\ell} \cdot 2^{\ell - 1} \cdot \cos^{\ell - 1}(\phi)\sin^{\ell}(\alpha_1)\sin^{\ell}(\alpha_2),
\]
which completes the proof. 
\end{proof}

\begin{lemma}\label{L:four}
The polynomial $P(\cos(\phi))$ evaluates to, respectively, $\cos(2 \ell(\alpha_1 + \alpha_2))$ and $\cos(2\ell(\alpha_1 - \alpha_2))$ at $\phi = 0$ and $\phi = \pi$.
\end{lemma}

\begin{proof}
We will use formula \eqref{eq:parametrize-Gamma-1}, which defines $P(\cos(\phi))$ for $\phi \in (0,\pi)$, to show that $P(\cos(\phi))$ limits to, respectively, $\cos(2\ell(\alpha_1+\alpha_2))$ and $\cos(2\ell(\alpha_1-\alpha_2))$ as $\phi \to 0$ and $\phi \to \pi$. We first calculate the limit as $\phi \to 0$. Once we eliminate the factors $\sin^2(\phi)$ as in Lemma \ref{L:one}, the calculation reduces to evaluating the expression 
\[
\frac{A_1 \cdot B_1 + A_2 \cdot B_2 + A_3 \cdot B_3}{\sin^2(\alpha_2)}
\]
at $\phi = 0$. Here, we use the notation from the proof of Lemma \ref{L:two}. It is easy to see that $A_1$, $A_2$, $A_3$ evaluate to 
\[
\begin{aligned}
A_1 & = -  \sin(\alpha_1)\sin(\alpha_2)\cos(\alpha_1+\alpha_2), \\
A_2 & =  \quad \cos(\alpha_1)\sin(\alpha_2)\sin(\alpha_1+\alpha_2), \\
A_3 & =  \quad \sin(\alpha_1)\sin(\alpha_2)\sin(\alpha_1+\alpha_2).	
\end{aligned}
\]
To evaluate $B_1, B_2$ and $B_3$, we will keep track of $\sin(\phi)$ in the formula $Q_1 = (X_1X_2)^{\ell}\cdot P_1 \cdot (X_1X_2)^{-\ell}$ while setting $\cos(\phi)$ equal to one. An induction on $\ell$ can be used to show that
\begin{equation}\label{eq:X1X2_power}
(X_1X_2)^{\ell}\, =\, q^{\ell} \,+\, \sin(\phi)\sin(\alpha_1)\,e^{\i \alpha_1}\left(q^{-\ell}\, \frac{1 - q^{2\ell}}{1 - q^2}\right)\j\, +\, \dots
\end{equation}
where $q = e^{\i(\alpha_1+\alpha_2)}$ and the dots stand for higher degree polynomials in $\sin(\phi)$. Using
the identity
\[
\frac{1 - q^{2\ell}}{1 - q^2}\, =\, e^{\i(\ell-1)(\alpha_1+\alpha_2)}\cdot\frac{\sin\left(\ell\left(\alpha_1+\alpha_2\right)\right)}{\sin(\alpha_1+\alpha_2)},
\]

\smallskip\noindent
the above formulas can be written in trigonometric form. It now follows from the formula  
\[
Q_1\, =\, (X_1X_2)^{\ell}\cdot P_1 \cdot (X_1X_2)^{-\ell}\, =\, \begin{pmatrix}
\;\qquad\quad\; B_1 \\
\;\sin(\phi) \cdot B_2 \\
\;\sin(\phi) \cdot B_3
\end{pmatrix}
\]
that, when $\phi = 0$, we have
\begin{align*}
B_1 &= 1 \\
B_2 &= \cos(2\ell(\alpha_1+\alpha_2)) + 2 \sin(\alpha_1)\big(-\sin(\alpha_2)\cos(\ell(\alpha_1+\alpha_2))\hspace{1in} \\ & \tag*{$ + \cos(\alpha_2)\sin(\ell(\alpha_1+\alpha_2))\big)\,\dfrac{\sin\left(\ell\left(\alpha_1+\alpha_2\right)\right)}{\sin(\alpha_1+\alpha_2)}$\qquad} \\
B_3 &= \sin(2\ell(\alpha_1+\alpha_2))-2\sin(\alpha_1)\big(\cos(\alpha_2)\cos(\ell(\alpha_1+\alpha_2)) \\ &
\tag*{$ + \sin(\alpha_2)\sin(\ell(\alpha_1+\alpha_2))\big)\,\dfrac{\sin\left(\ell\left(\alpha_1+\alpha_2\right)\right)}{\sin(\alpha_1+\alpha_2)}$\qquad}
\end{align*}
Finally, a tedious but straightforward trigonometric calculation using the above formulas shows that
\[
A_1 B_1 + A_2 B_2 + A_3 B_3 \, = \, \sin^2(\alpha_2)\cdot \cos(2\ell(\alpha_1+\alpha_2)),
\]
which immediately implies that $P(\cos(\phi))$ limits to $\cos(2\ell(\alpha_1+\alpha_2))$ as $\phi \to 0$. The calculation of the limit as $\phi \to \pi$ is similar.
\end{proof}

Recall that the Chebyshev polynomial of the first kind $T_m (x)$ is the unique polynomial of degree $m$ satisfying the formula $T_m (\cos (\psi)) = \cos (m\psi)$ for $m = 0, 1, 2 \ldots$ The following theorem is the main result of this section. 

\begin{theorem}\label{T:chebyshev}
Let $T_{2\ell} (x)$ be the Chebyshev polynomial of the first kind of degree $2\ell$. Then the formula \eqref{eq:parametrize-Gamma-1} is equivalent to
\begin{equation}\label{E:chebyshev}
\cos (\theta)\; = \; T_{2 \ell} \big( \cos(\alpha_1) \cos(\alpha_2) - \cos(\phi) \sin(\alpha_1) \sin(\alpha_2) \big).
\end{equation}
\end{theorem}

\begin{proof}
The solutions of the equation $P(\cos(\phi)) = 1$ with $0 < \phi < \pi$ are precisely the intersection points of $\widehat{\Lambda}_2^{\alpha, c}$ and $\widehat{\Gamma}_{\beta}^{\alpha,c}$ in the pillowcase $\widehat{H}_2^{\alpha, c}$. According to \eqref{E:presentation}, these correspond to the conjugacy classes of irreducible representations $\rho: \pi_1 (S^3 \setminus L_{\ell}) \to \SU(2)$, which we described as the values of $\phi$ solving equation \eqref{eq:solution-phi}. Note that the function $P(\cos(\phi))$ achieves its absolute maximum at all such $\phi$ hence we also know that $P'(\cos(\phi)) = 0$.

Let us first assume that the values of $(\alpha_1,\alpha_2)$ are in a sufficiently small neighborhood of $(\pi/2,\pi/2)$ chosen so that equation \eqref{eq:solution-phi} has the maximal possible number of solutions, which is $\ell - 1$. Consider the polynomial 
\[
R(x)\, = \, P(x) - T_{2\ell} \big( \cos(\alpha_1) \cos(\alpha_2) - x\sin(\alpha_1) \sin(\alpha_2) \big).
\] 
This polynomial has $\ell - 1$ roots $x = \cos(\pi m/\ell)$, $m = 1,\ldots, \ell -1$, each of multiplicity at least two. In addition, it has roots $x = \pm 1$ by Lemma \ref{L:four}. On the other hand, the degree of $R(x)$ is at most $2\ell$ by Lemma \ref{L:two}, while the leading coefficients of $P(x)$ and $T_{2\ell} \big( \cos(\alpha_1) \cos(\alpha_2) - x\sin(\alpha_1) \sin(\alpha_2) \big)$ match by Lemma \ref{L:two}. Therefore, the degree of $R(x)$ is at most $2\ell - 1$ so $R(x)$ must vanish.

To finish the proof, we observe that both polynomials $P(x)$ and $T_{2\ell} \big( \cos(\alpha_1) \cos(\alpha_2) - x\sin(\alpha_1) \sin(\alpha_2) \big)$ are analytic functions in $(\alpha_1, \alpha_2)$. Since we proved that they equal each other in an open neighborhood of $(\pi/2,\pi/2)$, they must equal each other for all $(\alpha_1, \alpha_2)$.
\end{proof}

\subsection{Transversality}
In general, $\widehat{\Gamma}_{\beta}^{\alpha,c}$ need to be perturbed for the intersection number \eqref{E:def} to make sense. However, in the case of $(2,2\ell)$--torus links, no perturbation is necessary as the intersection is automatically transversal.
	
\begin{proposition}\label{P:transverse}
Let $\beta = (\sigma_1)^{2\ell}$ be the braid whose closure is the $(2,2\ell)$--torus link with $\ell \neq 0$. Then the intersection of $\widehat{\Gamma}_{\beta}^{\alpha,c}$ and $\widehat{\Lambda}_2^{\alpha, c}$ in $\widehat{H}_2^{\alpha, c}$ is transversal, and all the intersections points contribute with the same sign into the algebraic count.
\end{proposition}
	
\begin{proof}
We wil assume that $\ell > 0$ since the case of $\ell < 0$ can be handled similarly. The curves $\widehat{\Gamma}_{\beta}^{\alpha,c}$ and $\widehat{\Lambda}_2^{\alpha, c}$ intersect exactly at the points $(\phi,0)$, where $\phi$ solves the equation
\[
\cos(0) = 1 = T_{2 \ell} \big( \cos(\alpha_1) \cos(\alpha_2) - \cos(\phi) \sin(\alpha_1) \sin(\alpha_2) \big),
\]
which we have shown happens exactly when
\[
\cos(\alpha_1) \cos(\alpha_2) - \cos(\phi) \sin(\alpha_1) \sin(\alpha_2) = \cos \left( \frac{m\pi}{\ell} \right)\;\; \text{for} \;\;  m = 1, 2, \dots, \ell - 1.
\]
The curve $\widehat{\Gamma}_{\beta}^{\alpha,c}$ is smooth near each of the intersection points. It is parameterized by $(\phi,\theta)$, where $\theta$ is a smooth function of $\phi$ found by solving the equation \eqref{E:param}. The intersection points give local maxima of the Chebyshev polynomial $T_{2\ell}$ hence the curve $\widehat{\Gamma}_{\beta}^{\alpha,c}$ must be decreasing near these points as shown in \autoref{fig:graph-example}. Therefore, all the intersection numbers will be of the same sign once we prove that the intersections are transversal.
		
\begin{figure}[ht!]
\centering
\includegraphics[width=0.3\textwidth]{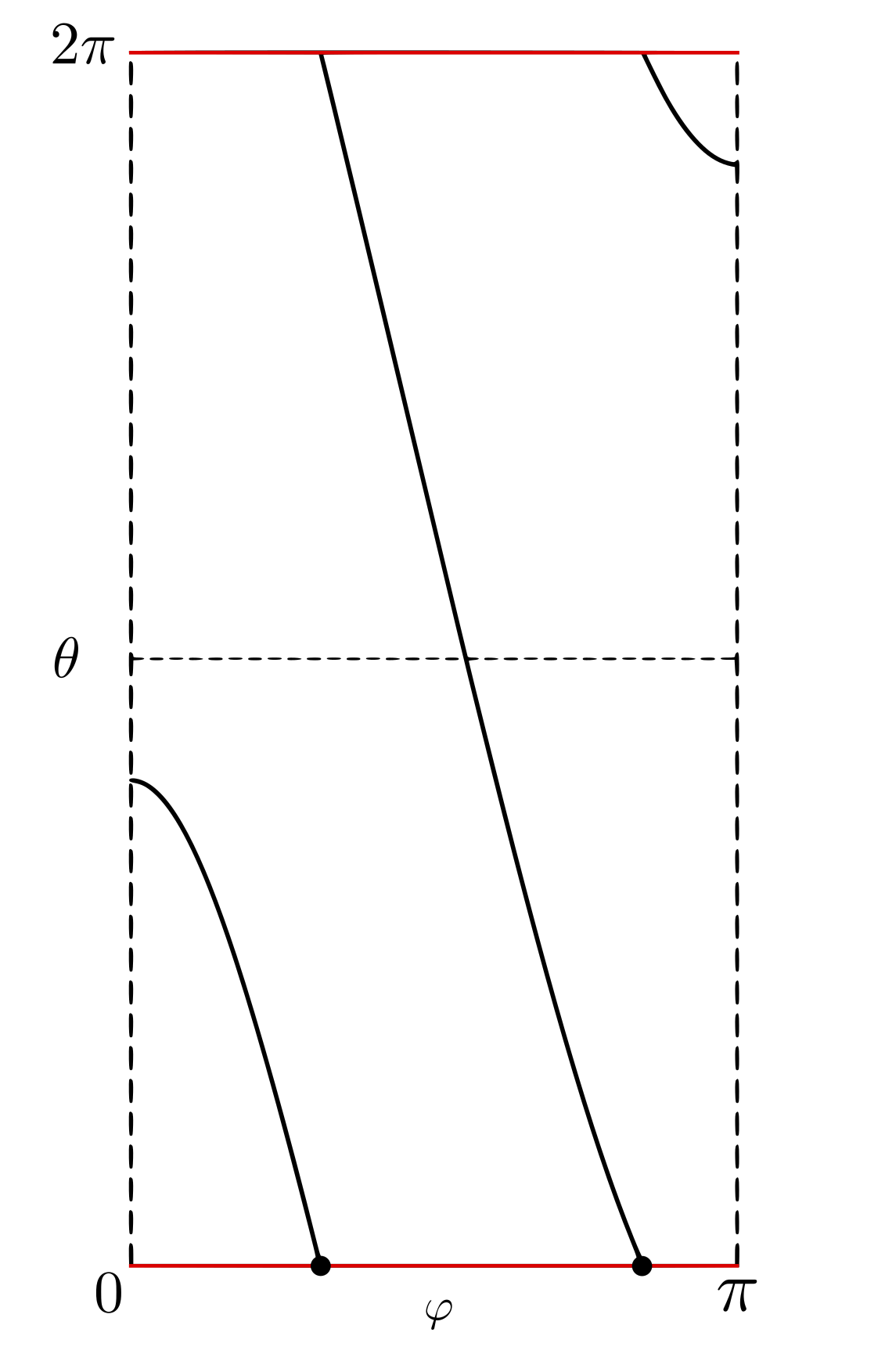}
\caption{An example of $\widehat{\Gamma}_{\beta}^{\alpha,c}$}
\label{fig:graph-example}
\end{figure}
		
To prove transversality, it suffices to show that the derivative of $\theta$ with respect to $\phi$ is not zero near the intersection points. Differentiating the formula \eqref{E:chebyshev} and keeping in mind that $T_n' (x) = n U_{n-1} (x)$, where $U_{n-1}(x)$ is the Chebyshev polynomial of the second kind of degree $n - 1$, we obtain
\[
\frac{d \theta}{d \phi}  = - 2 \ell\cdot \sin(\alpha_1) \sin(\alpha_2) \sin(\phi)\, \frac{U_{2 \ell - 1} (x)}{\sqrt{1 - T_{2 \ell}^2 (x)}}.
\]
We wish to show that, as $x\, \to \, \cos(\pi m/{\ell})$ from the left, the limit of the right hand side of this equation is non-zero. This is true because of the well known fact that each $\cos(\pi m/{\ell})$ is a simple root of $U_{2 \ell - 1}(x)$ and a double root of $1 - T_{2 \ell}^2 (x)$.
\end{proof}

\subsection{Fixing the overall sign}\label{sec:sign}
An immediate corollary of Proposition \ref{P:solution-phi} and Proposition \ref{P:transverse} is that, up to an overall sign, the B{\'e}nard--Conway invariant $h_{L_{\ell}}(\alpha)$ of the $(2,2\ell)$--torus link equals the number of integers $m \in \{1,\ldots, |\ell| -1\}$ such that 
\[
\cos(\pi m/\ell) \in [\,\cos(\alpha_1 + \alpha_2), \cos(\alpha_1 - \alpha_2)\,].
\]

To determine the sign, we need to get specific about the orientations of $\widehat{\Gamma}_{\beta}^{\alpha,c}$, $\widehat{\Lambda}_2^{\alpha, c}$ and $\widehat{H}_2^{\alpha, c}$ and to compute at least one intersection number explicitly. We will only compute the intersection number in the special case of $\alpha = (\alpha_1, \alpha_2) = (\pi/2, \pi/2)$. The calculation for general $\ell > 0$ and $\alpha$ will be similar, and the case of negative $\ell$ will be addressed later in this section. The following argument is a modification of the argument of Boden and Herald \cite {Boden-Herald}.

In the case at hand, $\widehat{\Gamma}_{\beta}^{\alpha,c}$ intersects $\widehat{\Lambda}_2^{\alpha, c}$ at one point $(\phi, \theta) = (\pi / 2, 0)$, which corresponds to the conjugacy class of $(\j, \i, \j, \i)$ in $\widehat{H}_2^{\alpha, c}$. Consider the function $f: R_2^{\alpha, c} \times R_2^{\alpha, c} \to \SU(2)$ given by $f(X_1, X_2, Y_1, Y_2) = X_1 X_2 Y_2^{- 1} Y_1^{- 1}$ and the function $g: (0, \pi)\, \times\, [0, 2 \pi) \to R_2^{\alpha, c}\times R_2^{\alpha, c}$ given by
\[
g(\phi, \theta) = (\i e^{- \k \phi}, \i, i e^{- \k (\phi - \theta)}, \i e^{\k \theta}).
\]
The latter is exactly the quadruple $(e^{\alpha_1 P_1}, e^{\alpha_2 \i}, e^{\alpha_1 Q_1}, e^{\alpha_2 Q_2})$ used to parameterize the pillowcase $\widehat{H}_2^{\alpha, c}$, with $\alpha_1 = \alpha_2 = \pi / 2$ substituted in the equation. Notice that $\widehat{H}_2^{\alpha, c} $ is the quotient of $f^{-1}(1)$ by conjugation. We will orient $f^{-1}(1)$ and $\widehat{H}_2^{\alpha, c}$ by applying the base--fiber rule. 

First, we will consider the map $g$. Two tangent vectors that span the tangent space to $\widehat{H}_2^{\alpha, c}$ at $(\j, \i, \j, \i)$ are given by
\[
\begin{aligned}
u_1 & = \left. \frac{\partial g}{\partial \phi}\, \right\vert_{\,(\phi,\theta)\, =\, ({\pi}/{2}, 0)} =\, (- \i, 0, - \i, 0), \\
u_2 & = \left. \,\frac{\partial g}{\partial \theta}\, \right\vert_{\,(\phi,\theta)\, =\, ({\pi}/{2}, 0)} =\, (0, 0, \i, - \j).
\end{aligned}
\]

\smallskip\noindent
The tangent space to the orbit through $(\j, \i, \j, \i)$ is spanned by the vectors 
\[
\begin{aligned}
v_1 & = \left. \frac{\partial}{\partial t} e^{\i t}\, (\j, \i, \j, \i)\, e^{- \i t}\;  \right\vert_{\,t = 0}  = (2 \k, 0, 2 \k, 0), \\
v_2 & = \left. \frac{\partial}{\partial t} e^{\j t}\, (\j, \i, \j, \i)\, e^{- \j t}\,  \right\vert_{\,t = 0}  = (0, - 2 \k, 0, - 2 \k), \\
v_3 & = \left. \frac{\partial}{\partial t} e^{\k t} (\j, \i, \j, \i) e^{- \k t} \,  \right\vert_{\,t = 0}   = (- 2 \i, 2 \j, - 2 \i, 2 \j).
\end{aligned}
\]
The vectors $\lbrace u_1, u_2, v_1, v_2, v_3 \rbrace$ form a basis of $T_{(\j, \i, \j, \i)}f^{- 1}(1)$. Complete it to a basis for $T_{(\j, \i, \j, \i)} \left(R_2^{\alpha, c} \times R_2^{\alpha, c}\right)$ using the vectors $\lbrace w_1, w_2, w_3 \rbrace$, where we can choose $w_1 = (\k, 0, 0, 0)$, $w_2 = (0, \k, 0, 0)$, and $w_3 = (0, \j, 0, 0)$. Notice that the orientation of the ordered triple $\lbrace df(w_1), df(w_2), df(w_3) \rbrace = \lbrace \i, - \j, - \k \rbrace$ is consistent with that of the standard basis of Lie algebra $\mathfrak{su}(2)$. 

The oriented bases $\lbrace w_1, w_2, w_3, u_1, u_2, v_1, v_2, v_3 \rbrace$ and $\{ (\i, 0, 0, 0), (\k, 0, 0, 0), (0, \j, 0, 0), \allowbreak (0, - \k, 0, 0), (0, 0, \i, 0), (0, 0, \k, 0), (0, 0, 0, \j), (0, 0, 0, - \k)\}$ of $T_{(\j, \i, \j, \i)} \left(R_2^{\alpha, c} \times R_2^{\alpha, c}\right)$ are related by the matrix
\[
\left(
\begin{array}{rrrrrrrr}
0 & 0 & \;\; 0 & - 1 & \;\; 0 & \;\; 0  & \;\; 0 & - 2 \\
1 & 0 & 0 & 0 & 0 & 2 & 0 & 0 \\
0 & 0 & 1 & 0 & 0 & 0 & 0 & 2 \\
0 & - 1 & 0 & 0 & 0 & 0 & 2 & 0 \\
0 & 0 & 0 & - 1 & 1 & 0 & 0 & - 2 \\
0 & 0 & 0 & 0 & 0 & 2 & 0 & 0 \\
0 & 0 & 0 & 0 & - 1 & 0 & 0 & 2 \\
0 & 0 & 0 & 0 & 0 & 0 & 2 & 0
\end{array}
\right)
\]

\medskip\noindent
of determinant is $- 8$. This implies that the ordered pair $\lbrace u_2, u_1 \rbrace$ gives a positively oriented basis in $T_{(\j, \i, \j, \i)}\, \widehat{H}_2^{\alpha, c}$. 

Next, consider the one--variable parameterizations $\widehat{\Lambda}_2^{\alpha, c} = \{(\i e^{- \k \phi}, \i, i e^{- \k \phi}, \i)\}$ and $\widehat{\Gamma}_{\beta}^{\alpha,c} = \{(\i e^{- \k \phi}, \i, (\i e^{- \k \phi}, \i))\,\sigma_1^4\}$, where $(X_1, X_2)\,\sigma_1 = (X_1 X_2 X_1^{- 1}, X_1)$ is the homeomorphism induced by the braid. Compute the tangent vector 
\[
\psi_1 = \left. \frac{\partial }{\partial \phi} (\i e^{- \k \phi}, \i, i e^{- \k \phi}, \i) \right\vert_{\phi = \frac{\pi}{2}}  = (- \i, 0, - \i, 0)
\]
to $\widehat{\Lambda}_2^{\alpha, c}$ and the tangent vector 
\[
\psi_2 = \left. \frac{\partial }{\partial \phi} (\i e^{- \k \phi}, \i, \sigma_1^4 (\i e^{- \k \phi}, \i)) \right\vert_{\phi = \frac{\pi}{2}}  = (- \i, 0, - 5 \i, 4 \j)
\]
to $\widehat{\Gamma}_{\beta}^{\alpha,c}$. The bases $\lbrace \psi_1, \psi_2 \rbrace$ and $\lbrace u_2, u_1 \rbrace$ in the tangent space $T_{(\j, \i, \j, \i)}\, \widehat{H}_2^{\alpha, c}$ are related by the matrix
\[
\left(\begin{array}{rr}
0 & - 4 \\
1 & \quad 1
\end{array}
\right).
\]
Since the determinant of this matrix is positive, we conclude that the intersection number $h_L(\alpha)=\langle \widehat{\Lambda}_2^{\alpha, c}, \widehat{\Gamma}_{\beta}^{\alpha,c} \rangle_{\widehat{H}_2^{\alpha, c}}$  equals $+1$.
		

\begin{corollary}\label{cor:representation-invariant}
For any integer $\ell > 0$ and for any choice of $(\alpha_1,\alpha_2)$ away from the set \eqref{E:roots}, the B{\'e}nard--Conway invariant $h_{L_{\ell}}(\alpha)$ of the $(2,2\ell)$--torus link is well defined and is equal to the number of integers $m \in \{1,\ldots, |\ell| -1\}$ such that 
\[
\cos(\pi m/\ell) \in [\,\cos(\alpha_1 + \alpha_2), \cos(\alpha_1 - \alpha_2)\,].
\]
\end{corollary}

\begin{figure}[!ht]
\centering
\includegraphics[width=0.9\textwidth]{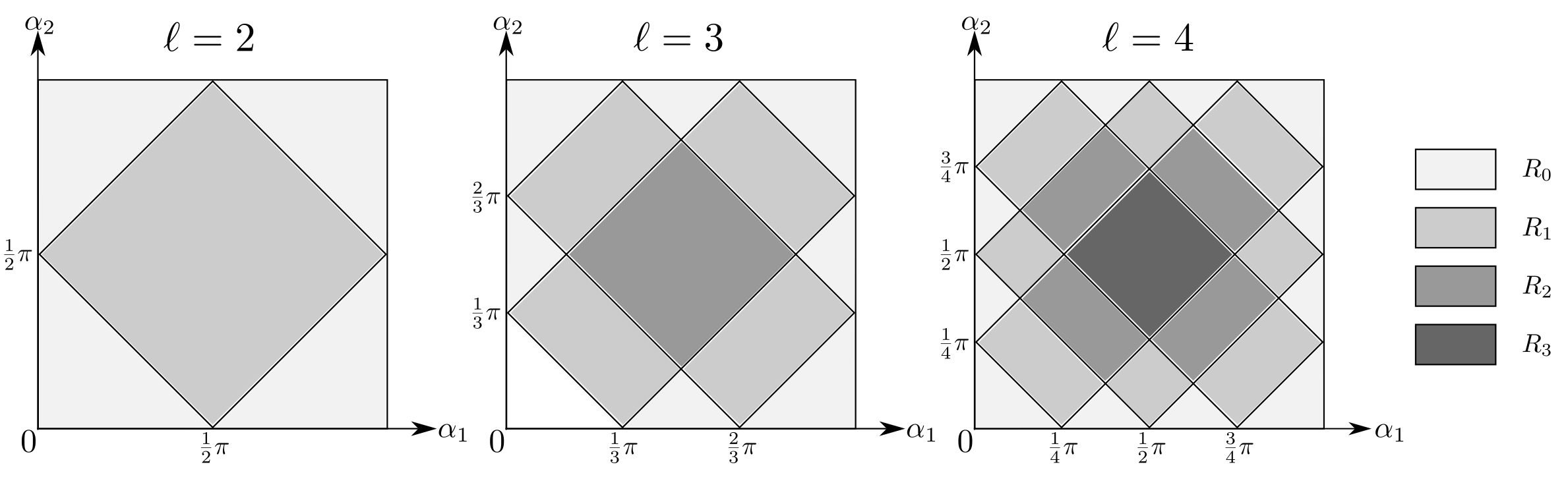}
\caption{Regions $R_m$ for small $\ell$}
\label{fig:Rm}
\end{figure}
		
For $\ell = 1$, we obviously obtain $h_{L_1}(\alpha) = 0$, as established by B{\'e}nard and Conway \cite{Benard-Conway}. For the first few $\ell \ge1$, the invariant $h_{L_{\ell}}(\alpha)$ equals $m$ in each of the respective diamond--shaped regions $R_m$ depicted in \autoref{fig:Rm}. The boundaries of $R_m$ consist of the values of $(\alpha_1,\alpha_2)$ from the set \eqref{E:roots}, where the multivariable Alexander polynomial of $L_{\ell}$ vanishes and hence the B{\'e}nard--Conway invariant is not defined. 
		
We will finish this section by computing the invariant $h_{L_{-\ell}} (\alpha)$ with $\ell > 0$. The link $L_{-\ell}$ is the closure of the braid $\sigma_1^{-2\ell}$ and the above calculation for $\ell = 2$ and $\alpha_1 = \alpha_2 = \pi/2$ can be repeated word for word. The only essential change occurs in parameterizing $\widehat{\Gamma}_{\beta}^{\alpha,c} = \{(\i e^{- \k \phi}, \i, (\i e^{- \k \phi}, \i)\,\sigma_1^{-4})\}$, which produces the tangent vector
\[
\psi_2 = \left. \frac{\partial}{\partial \phi} (\i e^{- \k \phi}, \i, \sigma_1^{- 4}(\i e^{- \k \phi}, \i))\,\right\vert_{\phi = \frac{\pi}{2}} = \, (- \i, 0, 3 \i, - 4 \j).
\]
The basis $\lbrace \psi_1, \psi_2 \rbrace$ is now related to the basis $\lbrace u_2, u_1 \rbrace$ by the matrix 
\[
\begin{pmatrix}
0 & 4 \\
1 & 1
\end{pmatrix}
\]
of negative determinant. This brings us to the following conclusion.

\begin{corollary}\label{C:h-sign}
For any integer $\ell > 0$ and for any choice of $(\alpha_1,\alpha_2)$ away from the set \eqref{E:roots} the invariants $h_{L_{\ell}}(\alpha)$ and $h_{L_{-\ell}}(\alpha)$ are well defined and related by the formula 
\[
h_{L_{-\ell}}(\alpha)\; = \, -h_{L_{\ell}} (\alpha).
\]
\end{corollary}

\section{Multivariable signature for torus links}\label{sec:six}
The goal of this section is to calculate the Cimasoni--Florens signature of the $(2, 2 \ell)$--torus link $L_{\ell}$, together with the quantity $- 1/2\cdot (\sigma_{L_{\ell}} (\omega_1, \omega_2) + \sigma_{L_{\ell}}(\omega_1, \omega_2^{- 1}))$, and to finish the proof of \thref{thm:main-theorem-introduction}.


\subsection{The positive linking number}
We will start by computing the Cimasoni--Florens signature of the 2-colored link $L_{\ell}$ with $\ell > 0$. The signature $\sigma_{L_{\ell}}(\omega_1,\omega_2)$ in question is the signature of the matrix $H_{\ell}(\omega)\, =\, (1 - \overline{\omega}_1) (1 - \overline{\omega}_2)\,A_{\ell}(\omega_1, \omega_2)$, where 
\[
A_{\ell}(\omega_1, \omega_2) = A_{\ell}^{+, +} - \omega_1 A_{\ell}^{+, -} - \omega_2 A_{\ell}^{-, +} + \omega_1 \omega_2 A_{\ell}^{-, -}
\] 
as in formula \eqref{E:matrix H}. Using the C-complex shown in \autoref{fig:2-4-torus}, we can easily see that $A^{+, -}_{\ell}$ and $A^{-, +}_{\ell}$ are always $(\ell - 1) \times (\ell - 1)$ zero matrices; $A^{+, +}_{\ell}$ is the $(\ell - 1) \times (\ell - 1)$ matrix with $- 1$ on the main-diagonal and $1$ on the super-diagonal, and $A_{\ell}^{-, -} = (A^{+, +}_{\ell})^{\top}$. In particular, the matrices are empty in the case of $\ell = 1$. Thus, $H_{\ell}(\omega_1, \omega_2)$ is the $(\ell - 1) \times (\ell - 1)$ matrix with $(1 - \overline{\omega}_1) (1 - \overline{\omega}_2) (- 1 - \omega_1 \omega_2)$ on the main-diagonal, $(1 - \overline{\omega}_1) (1 - \overline{\omega}_2)$ on the super-diagonal, and $(1 - \overline{\omega}_1) (1 - \overline{\omega}_2)\, \omega_1 \omega_2$ on the sub-diagonal.

\begin{figure}[!ht]
\centering
\includegraphics[width=0.9\textwidth]{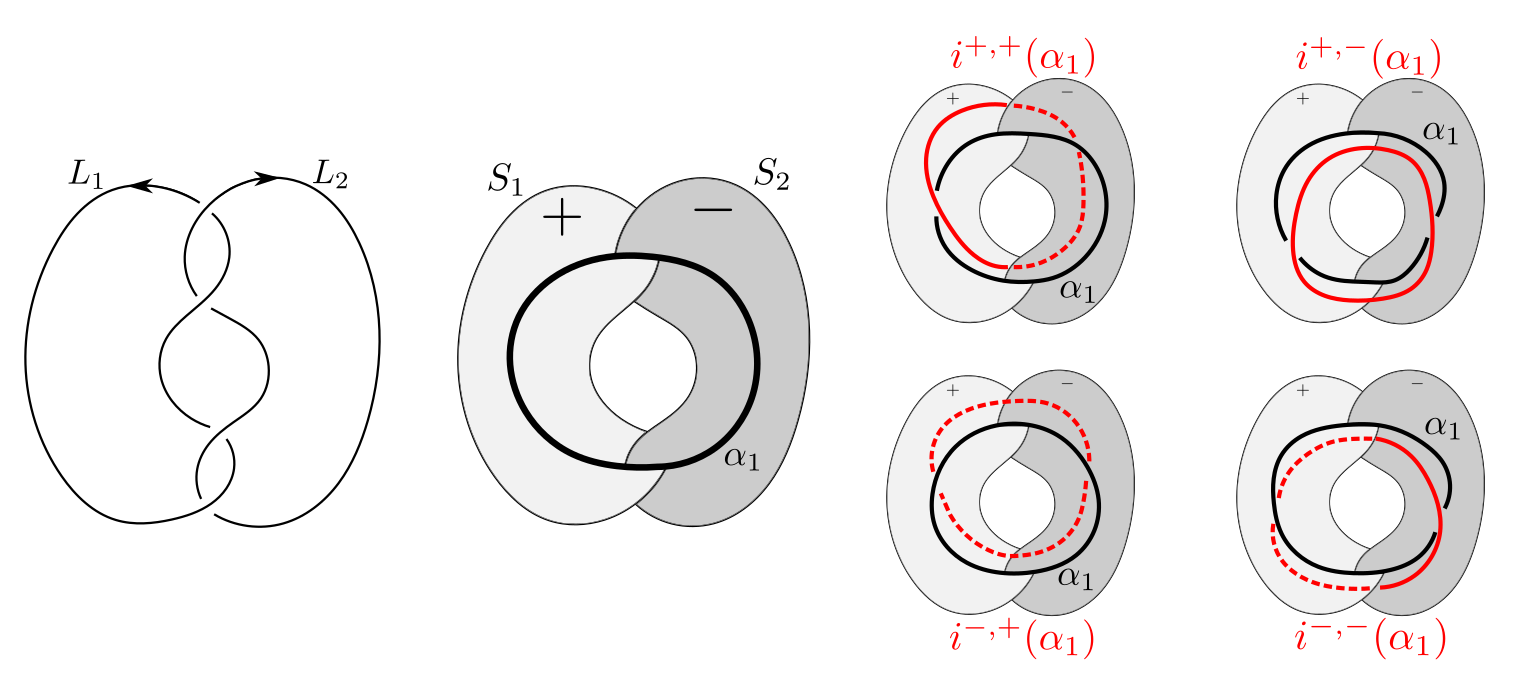}
\caption{A C-complex and push offs for the $(2, 4)$--torus link. All loops are oriented counter-clockwise}
\label{fig:2-4-torus}
\end{figure}

\begin{example}
A C-complex $S$ for the $(2, 4)$--torus link $L_2$ and the curves $i^{\epsilon}$ are shown in \autoref{fig:2-4-torus}. The complex $S$ is homotopy equivalent to a circle, with $H_1(S) = \mathbb{Z}$ generated by $[\alpha_1]$. Thus, for each $\epsilon = (\epsilon_1, \epsilon_2)$, the associated Seifert matrix $A_2^{\epsilon}$ is the $1 \times 1$ matrix $\big(\lk(i^{\epsilon}(\alpha_1), \alpha_1)\big)$. A direct counting gives that $A_2^{+, +} = A_2^{-, -} = \big(-1\big)$, $A_2^{+, -} = A_2^{-, +} = \big(0\big)$, and thus $H_2(\omega)$ is the $1 \times 1$ matrix $\big (1 - \overline{\omega}_1) (1 - \overline{\omega}_2) (- 1 - \omega_1\omega_2) \big)$. In particular, $\sigma_{L_2}(\omega) = \pm 1$ or $0$ depending on the sign of the real number $(1 - \overline{\omega}_1) (1 - \overline{\omega}_2) (- 1 - \omega_1\omega_2)$.
\end{example}

\begin{figure}[!ht]
\centering
\includegraphics[width=0.4\textwidth]{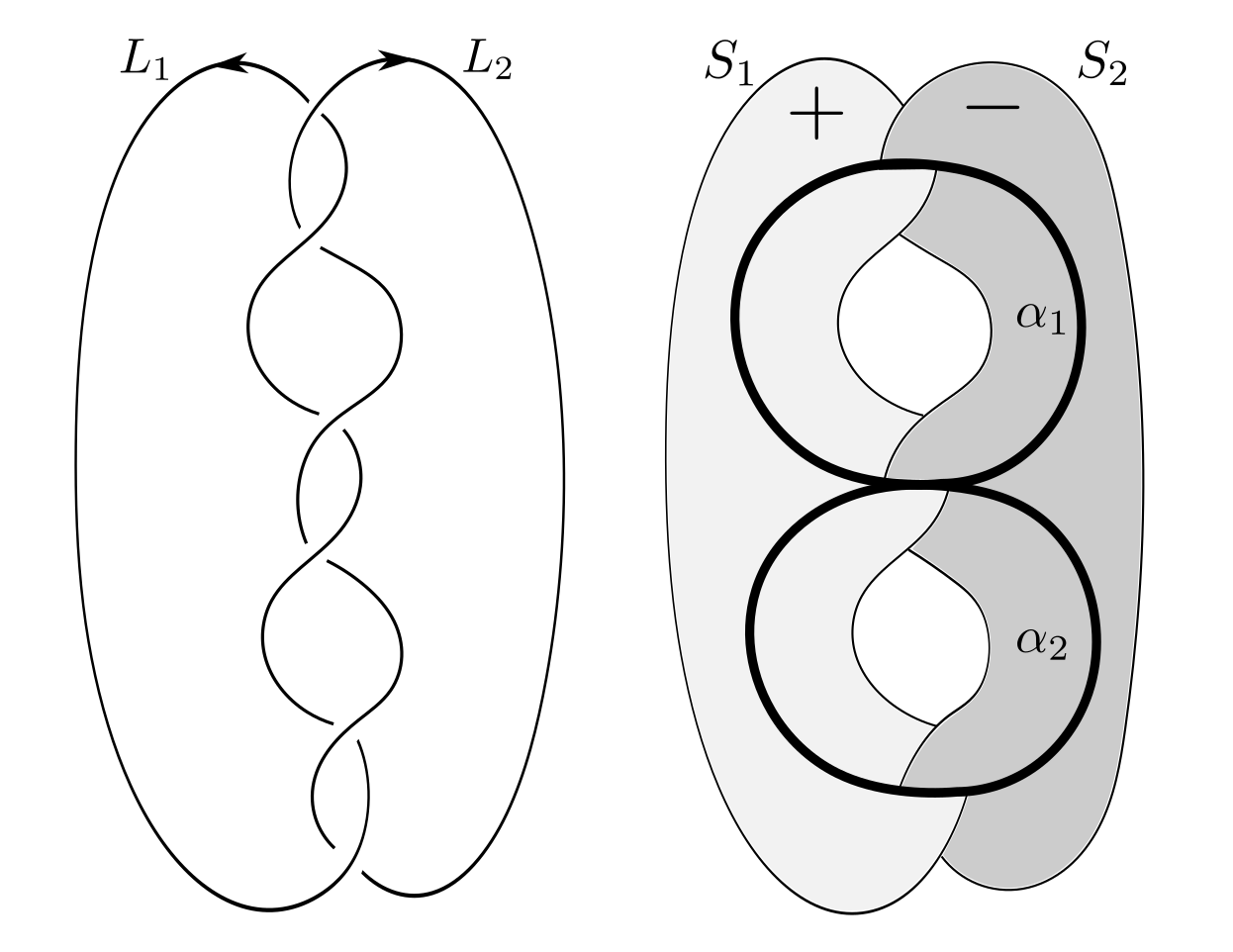}
\caption{The $(2, 6)$-torus link}
\label{fig:2-6-torus}
\end{figure}
		
\begin{example}
A C-complex $S$ of the $(2, 6)$-torus link $L_3$ is shown in \autoref{fig:2-6-torus}. The complex $S$ is homotopy equivalent to the one--point union $S^1 \vee S^1$, with $[\alpha_1]$ and $[\alpha_2]$ forming a basis of $H_1(S)$. By direct calculation,
\[
A_3^{+,+} = (A_3^{-, -})^{\top} = \left (\begin{array}{rr}
- 1 & 1 \\
0 & - 1
\end{array}\right) \quad\text{and}\quad A_3^{+, -} = A_3^{-, +} = \left( \begin{array}{rr}
\; 0 & \; 0 \\
0 & 0
\end{array}\right).
\]
Therefore,
\[
H_3(\omega) = \begin{pmatrix}
(1 - \overline{\omega}_1) (1 - \overline{\omega}_2) (- 1 - \omega_1 \omega_2) & (1 - \overline{\omega}_1) (1 - \overline{\omega}_2) \\
(1 - \omega_1) (1 - \omega_2) & (1 - \overline{\omega}_1) (1 - \overline{\omega}_2) (- 1 - \omega_1 \omega_2)
\end{pmatrix}.
\]
Notice that the top-left $(1 \times 1)$-minor of $H_3(\omega)$ is precisely $H_2(\omega)$.
\end{example}
		
We will now move on to the general case. For any $\ell > 0$ let $\delta_{\ell} = \det(H_{\ell}(\omega_1, \omega_2))$ so that $\delta_1 = 1$, the determinant of empty matrix, $\delta_2 = (1 - \overline{\omega}_1) (1 - \overline{\omega}_2) (- 1 - \omega_1 \omega_2)$ etc. In general, we have a recursive relation
\[
\delta_{m + 1} = (1 - \overline{\omega}_1) (1 - \overline{\omega}_2) (- 1 - \omega_1\omega_2)\,\delta_{m} - (1 - \overline{\omega}_1) (1 - \overline{\omega}_2) (1 - \omega_1) (1 - \omega_2)\,\delta_{m - 1}.
\]
Since $\omega_1, \omega_2 \in S^1 \setminus \lbrace 1 \rbrace$, we can write $\omega_1 = e^{2 \i \alpha_1}, \omega_2 = e^{2 \i \alpha_2}$ with $\alpha_1, \alpha_2 \in (0, \pi)$ and re-write the above formulas in trigonometric form so that $\delta_1 = 1$, $\delta_2 = 8 \sin(\alpha_1) \sin(\alpha_2) \cos(\alpha_1 + \alpha_2)$, plus the recursive relation
\[
\delta_{m + 1} = 8 \sin(\alpha_1) \sin(\alpha_2) \cos(\alpha_1 + \alpha_2)\, \delta_m - 16 \sin^2(\alpha_1) \sin^2(\alpha_2)\, \delta_{m - 1}.
\]
This recursive relation has a unique solution 
\begin{equation}\label{E:delta}
\delta_{m + 1} = 4^m \sin^m(\alpha_1) \sin^m(\alpha_2)\, U_m(\cos(\alpha_1 + \alpha_2)),
\end{equation}
where $U_m(x)$ is the $m$--th Chebyshev polynomial of the second kind, $U_m (\cos(\alpha_1 + \alpha_2)) \cdot \sin(\alpha_1 + \alpha_2) = \sin ( (m + 1) (\alpha_1 + \alpha_2))$. A direct calculation now shows that the signs of $\delta_{m + 1}$ are given by the formula
\[
\begin{aligned}
 \sign & (\delta_{m + 1})  = \sign(U_m(\cos(\alpha_1 + \alpha_2))) = \\
& \begin{cases}
\;\;\, 1 , & \alpha_1 + \alpha_2 \in \displaystyle\bigcup\limits_{0 \le i \le m, \, i \text{ even}}\left(\dfrac{i \pi}{m + 1}, \dfrac{(i + 1) \pi}{m + 1}\right) \textstyle{\bigcup} \left(2 \pi- \dfrac{(i + 1) \pi}{m + 1}, 2 \pi - \dfrac{i \pi}{m + 1}\right) \\
& \text{or } \alpha_1 + \alpha_2 = \pi \text{ and } m \text{ is even}, \\
\;\;\, 0, & \alpha_1 + \alpha_2 = \dfrac{j \pi}{m + 1},\; 1 \le j \le 2 m + 1,\; j \ne m + 1, \\
- 1, & \text{otherwise. }
\end{cases}
\end{aligned}
\]
			
\noindent
Since the top-left $m \times m$ minor of $H_{\ell}(\omega_1, \omega_2)$ for $0 < m < \ell$ is precisely $H_{m + 1}(\omega_1, \omega_2)$, we can use the Sylvester criterion to compute the signature by the formula
\[
\sigma_{L_{\ell}}(\omega_1, \omega_2) = \sign(H_{\ell}(\omega_1, \omega_2)) = \sign(\delta_2) + \sum_{i = 2}^{\ell - 1}\; \sign(\delta_i \cdot \delta_{i + 1}),
\] 
which is equivalent to the formula 
\[
\sigma_{L_{\ell}}(\omega_1, \omega_2) = \begin{cases}
\ell - 2 i - 1, & \dfrac{i \pi}{\ell} < \alpha_1 + \alpha_2 < \dfrac{(i + 1) \pi}{\ell}\; \text{ and }\; 0 \le\, i\, \le \ell - 1, \\
- 3 \ell + 2 i + 1, & \dfrac{i \pi}{\ell} < \alpha_1 + \alpha_2 <\dfrac{(i + 1) \pi}{\ell}\; \text{ and }\; {\ell} \le i \le 2 \ell - 1;
\end{cases}
\] 
see \autoref{fig:sign-signature} for a graphic depiction.

\begin{figure}[!ht]
\centering
\includegraphics[width=0.6\textwidth]{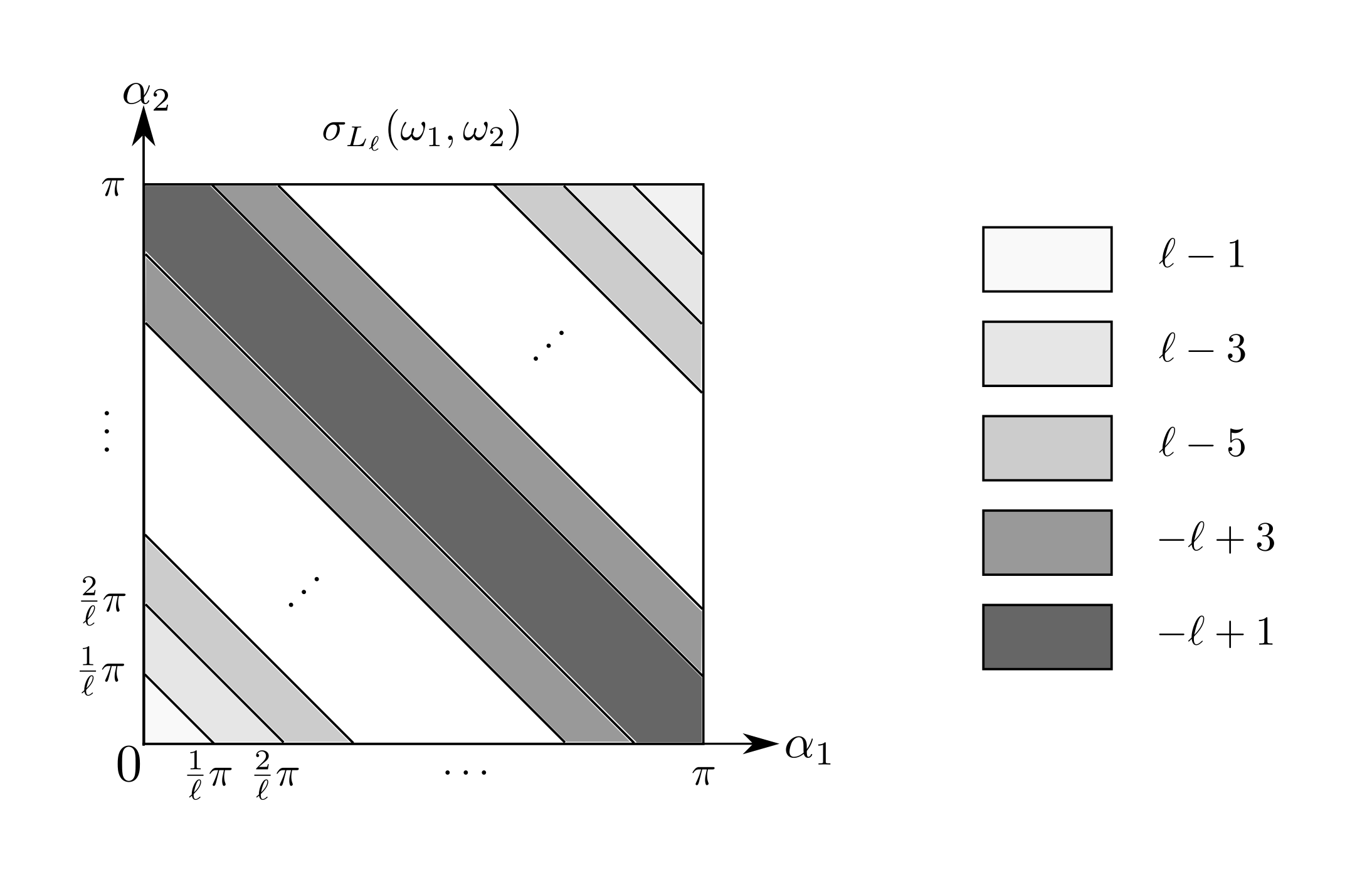}
\caption{Signature function $\sigma_{\ell} (\omega_1,\omega_2)$}
\label{fig:sign-signature}
\end{figure}
		
The quantity $\sigma_{L_{\ell}}(\omega_1, \omega_2^{- 1})$ is obtained from $\sigma_{L_{\ell}}(\omega_1, \omega_2)$ by a flip across the axis $\alpha_2 = \pi/2$, and the average of the two is then easily calculated. Comparing the answer with the calculation of Section \ref{sec:count-rep} and Corollary \ref{cor:representation-invariant}, we conclude that
\begin{equation}\label{E:formula}
h_{L_{\ell}}(\alpha) = h_{L_{\ell}}(\alpha_1, \alpha_2) = - \frac{1}{2}\,\big(\sigma_{L_{\ell}}(\omega_1, \omega_2) + \sigma_{L_{\ell}}(\omega_1, \omega_2^{- 1})\big).
\end{equation}

		
\subsection{The negative linking number}
For any positive integer $\ell$ and $\delta_{- \ell} = \det(H_{- \ell}(\omega_1, \omega_2))$, we calculate $\delta_{- 1} = 1$, $\delta_{- 2} = 8 \sin(\alpha_1) \sin(\alpha_2) \cos(\pi + \alpha_1 + \alpha_2)$, and we obtain the recursive relation
\[
\delta_{- m - 1} = 8 \sin(\alpha_1) \sin(\alpha_2) \cos(\pi + \alpha_1 + \alpha_2)\, \delta_{- m} - 16 \sin^2(\alpha_1) \sin^2(\alpha_2)\, \delta_{- m + 1}
\]
which has the unique solution
\[
\delta_{- m - 1} = 4^m \sin^m(\alpha_1) \sin^m(\alpha_2)\, U_m(\cos(\pi + \alpha_1 + \alpha_2)),
\]
which is obtained by a shift by $\pi$ of the formula \eqref{E:delta}. In the application of the Sylvester criterion, each of the summands flips the sign compared to the positive linking number case, that is, $\sign(\delta_{- 2}) = - \sign(\delta_2)$ and $\sign(\delta_{- i} \cdot \delta_{- i - 1}) = - \sign(\delta_i \cdot \delta_{i + 1})$ for each $i$ from $1$ to $\ell - 1$. Thus we obtain
\[
\sigma_{L_{- \ell}}(\omega_1, \omega_2) = - \sigma_{L_{\ell}}(\omega_1, \omega_2).
\]
Recall that $h_{L_{- \ell}}(\alpha) = - h_{L_{\ell}}(\alpha)$, see Corollary \ref{C:h-sign}. Therefore, the formula \eqref{E:formula} at the end of the last subsection extends to the following result. 
		
\begin{theorem}\thlabel{thm:base-step}
Let $L_{\ell}$ be the $(2, 2\ell)$--torus link with $\ell \neq 0$ and, for any choice of $(\alpha_1, \alpha_2) \in (0, \pi)^2$, let $(\omega_1, \omega_2) = (e^{2 \i \alpha_1},e^{2 \i \alpha_2})$. If the multivariable Alexander polynomial of $L_{\ell}$ satisfies $\Delta_{L_{\ell}} (\omega_1^{\epsilon_1}, \omega_2^{\epsilon_2}) \ne 0$ for all possible $\epsilon_1, \epsilon_2 = \pm 1$, then
\[
h_{L_{\ell}}(\alpha_1, \alpha_2) = - \frac{1}{2}\, (\sigma_{L_{\ell}}(\omega_1, \omega_2) + \sigma_{L_{\ell}}(\omega_1, \omega_2^{- 1})).
\]
\end{theorem}

	
\section{Proof of Theorems \ref{thm:main-theorem-introduction} and \ref{thm:second-theorem-introduction}}\label{sec:seven}		
According to \cite{Milnor1}, the linking number is a complete invariant of link homotopy for two-component links. Therefore, any two-component oriented link with linking number $\ell \ne 0$ can be obtained from the $(2, 2 \ell)$--torus link via a sequence of crossing changes within individual components. According to \thref{thm:base-step}, the equality \eqref{E:main} holds for all $(2, 2 \ell)$--torus links with $\ell \neq 0$, and according to \thref{thm:inductive-step}, the equality \eqref{E:main} remains true with each crossing change. Theorem \ref{thm:main-theorem-introduction} now follows.

In order to prove Theorem \ref{thm:second-theorem-introduction}, we will use the following formula of Cimasoni and Florens \cite[Proposition 2.8]{Cimasoni-Florens}, 
\[
\sigma_{L_1\,\cup\, L_2} \left(\, \omega_1, \omega_2^{-1} \right)\; = \; \sigma_{L_1\, \cup\, -L_2}\left(\, \omega_1, \omega_2 \right),
\]
where $-L_2$ stands for the component $L_2$ with reversed orientation. Together with the formula
\[
\sigma_{L_1\,\cup\, L_2} \left(\, \omega_1^{-1}, \omega_2^{-1} \right)\; = \; \sigma_{L_1\, \cup\, L_2}\big(\, \omega_1, \omega_2 \big),
\]
which easily follows from the definition of the Cimasoni--Florens signature, it implies that
\[
h_L (\alpha_1, \alpha_2) \; = \; - \frac{1}{2}\, \big(\sigma_{L_1\,\cup\,L_2}\, (\,\omega_1, \omega_2) + \sigma_{L_1\,\cup\,-L_2} (\omega_1, \omega_2) \big)
\]
is independent of the choice of orientation on the link $L$. If $\omega_1 = \omega_2 = \omega$, one can use the following formula of \cite[Proposition 2.5]{Cimasoni-Florens},
\[
\sigma_{L_1\,\cup\,L_2}\, (\omega)\; = \; \sigma_{L_1\,\cup\,L_2}\, (\omega,\omega) - \lk (L_1,L_2),
\]
relating the multivariable signature $\sigma_L (\omega,\omega)$ with the Levine--Tristram signature $\sigma_L (\omega)$, to obtain
\begin{align*}
h_L (\alpha, \alpha)\; & = \; - \frac{1}{2}\, \big(\sigma_{L_1\,\cup\,L_2}\, (\,\omega, \omega) + \sigma_{L_1\,\cup\,-L_2}\, (\omega, \omega) \big) \\
& = \; - \frac{1}{2}\, \big(\sigma_{L_1\,\cup\,L_2}\, (\,\omega) + \lk (L_1,L_2) + \sigma_{L_1\,\cup\,-L_2}\, (\omega) + \lk (L_1,-L_2) \big) \\
& = \; - \frac{1}{2}\, \big(\sigma_{L_1\,\cup\,L_2}\, (\,\omega) + \sigma_{L_1\,\cup\,-L_2}\, (\omega) \big).
\end{align*}
In the special case of $\omega = -1$, this implies that $h_L (\pi/2, \pi/2)$ equals minus the Murasugi signature of $L$. 

\begin{remark}
There does not appear to exist a name in the literature for the quantity $1/2\, \big(\sigma_{L_1\,\cup\,L_2}\, (\,\omega) + \sigma_{L_1\,\cup\,-L_2}\, (\omega) \big)$ nor for a similar quantity for links with more than two components when $\omega \ne -1$. Perhaps it should be called the equivariant Murasugi signature. 
\end{remark}

\end{sloppypar}


\begin{thebibliography}{999}


\bibitem{Benard-Conway}
L.~B{\'e}nard, A.~Conway,
\emph{A multivariable Casson--Lin type invariant}, Ann. Inst. Fourier (Grenoble) \textbf{70} (2020), 1029--1084


\bibitem{Boden-Harper}
H.~Boden, E.~Harper,
\emph{The SU(N) Casson--Lin invariants for links}, Pacific J. Math. \textbf{285}~(2016), 257--282

\bibitem{Boden-Herald}
H. Boden, C. Herald, \emph{The SU(2) Casson-Lin invariant of the Hopf link}, Pacific J. Math. \textbf{285}~ (2016), 283--288

\bibitem{Cimasoni}
D.~Cimasoni,
\emph{A geometric construction of the Conway potential function}, Comment. Math. Helv. \textbf{79} (2004), 124--146

\bibitem{Cimasoni-Florens}
D.~Cimasoni, V.~Florens, 
\emph{Generalzied Seifert surfaces and signatures of colored links}, Trans. Amer. Math. Soc. \textbf{360} (2008), 1223--1264

\bibitem{Daemi-Scaduto}
A.~Daemi, C.~Scaduto,
\emph{Equivariant aspects of singular instanton Floer homology}, Geometry Topol. (to appear). Preprint arXiv:1912.08982

\bibitem{Daemi-Scaduto-2}
A.~Daemi, C.~Scaduto,
\emph{Unoriented skein exact triangles in equivariant singular instanton Floer theory}. Preprint arXiv:2409.16390

\bibitem{Harper-Saveliev}
E.~Harper, N.~Saveliev,
\emph{A Casson--Lin type invariant for links}, Pacific J. Math. \textbf{248}~(2010), 139--154

\bibitem{Herald}
C.~Herald,
\emph{Flat connections, the Alexander invariant, and Casson's invariant}, Comm. Anal. Geom. \textbf{5} (1997), 93--120

\bibitem{Heusener-Kroll}
M.~Heusener, J.~Kroll,
\emph{Deforming abelian $SU(2)$-representations of knot groups}, Comment. Math. Helv. \textbf{73} (1998), 480--498

\bibitem{Kronheimer-Mrowka}
P.~Kronheimer, T.~Mrowka,
\emph{Khovanov homology is an unknot-detector}, Publ. Math. Inst. Hautes {\'E}tudes Sci. \textbf{113} (2011), 97--208

\bibitem{Lin}       
X.--S.~Lin, \emph{A knot invariant via representations spaces},
J. Differential Geom. \textbf{35} (1992), 337--357


\bibitem{Milnor1}
J. Milnor, \emph{Link groups}, Ann. of Math. \textbf{59}~(1954), 177--195

\bibitem{Milnor2}
J. Milnor, Singular Points of Complex Hypersurfaces. Princeton University Press, 1968

\bibitem{Murakami}
J.~Murakami, \emph{A state model for the multivariable Alexander polynomial}, Pacific J. Math. \textbf{157} (1993), 109--135

\bibitem{Murasugi}
K. Murasugi, \emph{On the signature of links}, Topology \textbf{9} (1970), 283--298

\bibitem{Torres}
G.~Torres, \emph{On the Alexander polynomial}, Ann. of Math. \textbf{57}~(1953), 57--89

\end{thebibliography}

\end{document}